\documentclass[12pt]{article}
\usepackage{amsmath}
\usepackage{amsfonts}
\usepackage{amsthm}
\usepackage{graphicx}
\usepackage{overpic}
\usepackage{a4wide}
\usepackage{amssymb} 
\usepackage{pictex}
\usepackage{rotating}  
\usepackage{color}

\definecolor{refkey}{rgb}{0,0,1}
\definecolor{labelkey}{rgb}{1,0,0}

\usepackage{comment}

\numberwithin{equation}{section}

\newtheorem{theorem}{Theorem}[section]
\newtheorem{proposition}[theorem]{Proposition}
\newtheorem{lemma}[theorem]{Lemma}
\newtheorem{corollary}[theorem]{Corollary}
\newtheorem{Definition}[theorem]{Definition}
\newenvironment{definition}{\begin{Definition}\rm}{\end{Definition}}
\newtheorem{Remark}[theorem]{Remark}
\newenvironment{remark}{\begin{Remark}\rm}{\end{Remark}}
\newtheorem{RHproblem}[theorem]{RH problem}

\usepackage{color}

\newcommand{\C}{\mathbb{C}}

\newcommand{\R}{\mathbb{R}}

\newcommand{\MM}{\mathcal M}

\renewcommand{\Re}{{\rm Re} \,}

\renewcommand{\hat}{\widehat}
\renewcommand{\tilde}{\widetilde}

\def\supp{\mathop{\mathrm{supp}}\nolimits}

\def \deg{\mbox{{\rm deg} }}
\def \capa{\mbox{{\rm cap}\,}}

\begin{document}
\title{Vector Energy and Large Deviation}
\author{T. Bloom, N. Levenberg, and F. Wielonsky}
\maketitle 
\begin{abstract}
For $d$ nonpolar compact sets $K_1,...,K_d\subset \C$, admissible weights $Q_1,...,Q_d$ and a positive semidefinite interaction matrix $C=(c_{i,j})_{i,j=1,...,d}$ with no zero column, we define natural discretizations of the weighted energy 
$$E_Q(\mu):=\sum_{i,j=1}^d c_{i,j}I(\mu_i,\mu_j) + 2\sum_{j=1}^d\int_{K_j} Q_j d\mu_j$$ of a $d-$tuple of positive measures $\mu=(\mu_1,...,\mu_d)\in \mathcal M_r(K)$ where $\mu_j$ is supported in $K_j$ and has mass $r_j$. We have an $L^{\infty}-$type discretization $W(\mu)$ and an $L^2-$type discretization $J(\mu)$ defined using a fixed measure $\nu=(\nu_1,...,\nu_d)$. This leads to a large deviation principle for a canonical sequence $\{\sigma_k\}$ of probability measures on $\mathcal M_r(K)$ if $\nu$ is a strong Bernstein-Markov measure. 
\end{abstract}
%
%%tableofcontents
%
\section{Introduction and main results} 

We prove a large deviations principle (LDP) which applies to the normalized counting measure of a random point in many multiple orthogonal polynomial ensembles, including Angelesco and certain Nikishin ensembles with compact supports. Our starting point is a very general vector energy setting first introduced in \cite{[GR]}, \cite{NS} and further studied in \cite{BKMW}, \cite{ESZ} and \cite{[HK1]} associated to $d$ compact sets $K_1,...,K_d\subset \C$, admissible weights $Q_1,...,Q_d$, and a positive semidefinite interaction matrix $C=(c_{i,j})_{i,j=1,...,d}$. We then define energy discretizations giving rise to the appropriate configuration space of points on the $d-$tuple of sets $K_1,...,K_d$. 

  Multiple orthogonal polynomials (MOPs) are a generalization of orthogonal polynomials in which the orthogonality is distributed among a number of orthogonality weights. They have been studied in connection with problems in analytic number theory, approximation theory and from the point of view of new special functions. In recent years MOPs have appeared in probability theory and certain models in mathematical physics coming from random matrices as MOPs can naturally give rise to ensembles of probability measures. This was first observed by Bleher and Kuijlaars \cite{[BK]} in the study of random matrix models with external source. Moreover, in the Gaussian case, the external source model is equivalent to a model involving non-intersecting Brownian motion. An excellent account of the recent developments in the application of MOPs with extensive references can be found in \cite{K1} or \cite{K2}. Generally the MOPs have been studied using Riemann-Hilbert methods.

In this paper we are primarily concerned with the almost sure convergence of a random point in an ensemble to an equilibrium measure (Corollary \ref{416}) and a large deviation principle (Theorem \ref{ldp}). We begin with the discretization of a general vector energy.
Two important special cases of the general ensembles we study in this paper are the Angelesco MOP case with interaction matrix $ C=(c_{i,j})_{i,j=1,\ldots,d}$ where $c_{i,i}=1$ and $c_{i,j}=1/2$ for $i\neq j$ and the Nikishin MOP case with interaction matrix given by
$c_{i,i}=1$, $c_{i,j}=-1/2$ if $|i-j|=1$, $c_{i,j}=0$ otherwise. These ensembles commonly arise from models in mathematical physics and they have a natural discretization -- the points represent eigenvalues of matrices or positions of particles (see \cite{K1}, \cite{K2}). In addition, $\beta$ ensembles of random matrices correspond to a $1\times 1$ interaction matrix consisting of a positive real number $\beta$ and hence they may also be considered as a special case of the ensembles considered here.

In the case of disjoint compact intervals of the real line, a LDP for Angelesco ensembles was established in \cite{bloomangel} using potential theory
and in \cite{ESZ} where an extension of  the method of Ben Arous-Guionnet \cite{BG, AGZ} was used. Recently, a LDP has also been obtained for the spectral measures of a non-centered Wishart matrix model, whose eigenvalue distribution can be described as a Nikishin ensemble in the presence of an external field on $\R_{+}$ and a constraint on $\R_{-}$, see \cite{[HK2]}. In this paper we use potential theory and polynomial inequalities to establish our results, valid for {\it nonpolar compacta in} $\C$. We first prove the almost sure convergence of a random point to the equilibrium measure and subsequently establish the LDP. This method shows  (Remark \ref{72}) that the rate function in the LDP  is independent of the measure used to define $ L^2$ norms as long as the measure satisfies a general condition, a {\it strong Bernstein-Markov property}.

The outline of the paper is as follows. In the next section we describe the vector energy minimization problems in the weighted and unweighted case. For clarity of exposition, we assume our compact sets are disjoint until section 8. The main idea is to give a discrete version of these energies $E(\mu)$ and $E_Q(\mu)$ following the ideas in \cite{bloomangel}. These discretizations are $L^{\infty}$ approximations and in order to develop the appropriate LDP, we need to introduce $L^2$ versions. This leads to notions of a (weighted) Bernstein-Markov property for vector measures which is the content of section 4. The utilization of a measure satisfying a strong (rational) Bernstein-Markov property is crucial for our approach to the LDP. To handle the case where some coefficients $c_{i,j}$ of $C$ are negative we need to extend the notion of Bernstein-Markov property from polynomials to rational functions. In Theorem \ref{strongrbm} we show that any nonpolar compact set in $\C$ admits a measure satisfying a strong rational Bernstein-Markov property. Here we need to appeal to a result from \cite{PELD} in $\C^n$ for $n>1$.

In particular, Proposition \ref{weightedtd} shows that the asymptotics of a sequence of (weighted) $L^2$ ``free energies'' are the same as their $L^{\infty}$ counterparts. To overcome a technical issue in the proof of our LDP in section 7, we must consider a non-admissible weighted problem, which shows up even in the scalar setting. We deal with this in section 5 using an approximation scheme with the aid of a deep result of Ancona \cite{ancona}. We define our $L^2$ and $L^{\infty}$ vector energy functionals $J$ and $W$ in section 6 culminating in the statement and proof of our LDP in section 7. Section 8 indicates cases where our results remain valid, including an LDP, for possibly intersecting sets $K_1,...,K_d$.

 \section{Vector equilibrium problems}\label{vep}

We begin with some potential-theoretic preliminaries in the scalar setting; i.e., associated to a single compact set. Let $Q$ be an {\it admissible} weight on a nonpolar compact set $K\subset \C$. This means $Q$ is lowersemicontinuous and finite on a set of positive logarithmic capacity; i.e., $\capa(\{z\in K: Q(z)<+\infty\})>0$. The usual weighted energy minimization problem is: 
$$\inf_{\mu \in \mathcal M(K)} \Big( I(\mu)+2\int_KQd\mu \Big)$$
where $\mathcal M(K)$ denotes the probability measures on $K$ and $I(\mu)$ is the standard logarithmic energy:
$$I(\mu)=\int_K \int_K \log\frac{1}{|z-t|}d\mu(z)d\mu(t)>-\infty.$$ 
We consider the slightly more general case where we minimize over $\mathcal M_r(K)$, the positive measures on $K$ of total mass $r>0$. We always have existence and uniqueness of a weighted energy minimizing measure $\mu^{K,Q}$. We write $\mu^K$ in the unweighted case ($Q\equiv 0$). Recalling that the logarithmic potential function of a measure $\mu$ is defined by 
$$U^{\mu}(z):=\int\log\frac{1}{|z-t|}d\mu(t),$$
we say $K$ is {\it regular} if $U^{\mu^K}$ is continuous. In the weighted case, 
there exists a constant $F$ such that the logarithmic potential $U:=U^{\mu^{K,Q}}$ satisfies
\begin{align*}
U(z) +Q& \geq F,\quad \text{q.e. }z\in K,\\
U(z) +Q& \leq F,\quad \forall z\in\supp(\mu^{K,Q})\\
\end{align*}
(``q.e.'' means off of a polar set). Indeed, one can, in analogy with the case $r=1$, define a weighted extremal function
$$V_{Q_{}}(z)=\sup\{g(z): \ g \in \mathcal L_r, ~g\leq Q_{}\text{ on }K\},$$
and $V_{Q}^{*}$, its uppersemicontinuous regularization, where $\mathcal L_r$ denotes the class of subharmonic functions in $\C$ of growth at most $r\log |z|$ as $|z|\to \infty$. Then 
$$V_Q^{*}=-U+F.$$

Let us now consider the vector case, where a $d$-tuple of nonpolar compact sets $K=(K_{1},\ldots,K_{d})$ and a $d$-tuple of admissible weights $Q=(Q_{1},\ldots,Q_{d})$ with $Q_{i}$ defined on $K_{i}$, $i=1,\ldots,d$, are given, along with a symmetric positive semidefinite interaction matrix
$$C:=(c_{i,j})_{i,j=1}^{d},$$
with no zero columns (or rows).  
Throughout, until section 8, we assume that the sets $K_{i}$, $i=1,\ldots,d$, are {\em pairwise disjoint}.
% {\blue We will always assume $K_1$ and $K_2$ are nonpolar.} 
The unweighted energy of a $d$-tuple of measures $\mu=(\mu_{1},\ldots,\mu_{d})$ is defined as
$$E(\mu):=\sum_{i,j=1}^{d}c_{i,j}I(\mu_{i},\mu_{j}),$$
where $I(\mu_{i},\mu_{j})$ is the mutual energy:
$$I(\mu_{i},\mu_{j})=\int\int \log\frac{1}{|z-t|}d\mu_{i}(z)d\mu_{j}(t).$$
Note that, with the above assumptions, $I(\mu_{i},\mu_{j})\in(-\infty,\infty)$ if $i\not = j$. 
The weighted energy of $\mu=(\mu_{1},\ldots,\mu_{d})$ is defined as
$$E_Q(\mu):=E(\mu)+2\sum_{i=1}^{d}\int Q_{i}d\mu_{i}.$$
We fix $r_1,\ldots,r_d>0$ and from now on we set 
\begin{equation*}
\mathcal M_r(K):=\{\mu=(\mu_{1},\ldots,\mu_{d}), \quad\mu_i\in \mathcal M_{r_i}(K_i), i=1,\ldots,d
\}.
\end{equation*}
We equip $\mathcal M_r(K)$ with the (component-wise) weak-* topology. If we need to keep track of the underlying interaction matrix $C$, we write a superscript $C$; e.g., 
$E^C$ and $E^C_Q$. Note since $C\geq 0$ and since $-\log$ and $Q$ are lowersemicontinuous functions, we have $E$ and $E_Q$ are lowersemicontinuous functionals on $\mathcal M_r(K)$ (see \cite[Chapter 5, Proposition 4.1]{NS} and \cite[Proposition 2.10]{BKMW} where the $K_{i}$ may intersect). 

%{\red \begin{remark} \label{normalize} For each measure $\mu=(\mu_{1},\mu_{2})\in \mathcal M_r(K)$ we can consider the associated measure $\tilde \mu=(\mu_{1}/r_1,\mu_{2}/r_2)\in \mathcal M(K)$ and it easily follows that, taking the symmetric positive semidefinite interaction matrix $$\tilde C =\begin{pmatrix} r_1^2c_{1,1} & r_1r_2c_{1,2}\\ r_1r_2c_{2,1} & r_2^2c_{2,2} \end{pmatrix},$$ we have $E^C(\mu)=\tilde E^{\tilde C}(\tilde \mu)$. Thus in many purely potential-theoretic considerations, we may assume $r_1=r_2=1$.\end{remark}} 

From Theorem 1.8 of \cite{BKMW}, it is known in the unweighted case there exists a unique minimizing $d$-tuple of measures for the energy $E$ over $\mu \in \mathcal M_r(K)$ (for a positive definite $C$, the result is also proven in \cite[Chapter 5]{NS}). We write this measure as $\mu^K=(\mu_1^K,\ldots,\mu_d^K)$ and $E(\mu^K)=E^*$; in the weighted case there exists a unique minimizing tuple of measures for the energy $E_Q$ and we write this measure as $\mu^{K,Q}=(\mu_1^{K,Q},\ldots,\mu_d^{K,Q})$ and $E_Q(\mu^{K,Q})=E_Q^*$. Moreover, if we introduce the partial potentials
$$U^{\mu}_{i}=\sum_{j=1}^{d}c_{i,j}U^{\mu_{j}},\qquad i=1,\ldots,d,$$ 
it is proved in  \cite[Theorem 1.8]{BKMW} that a measure $\mu$ minimizes the weighted energy $E_{Q}$ if and only if there exist constants $F_{1},...,F_d$ such that
\begin{align}\label{eq-var1}
U^{\mu}_{i}(z) +Q_i& \geq F_{i},\quad \text{q.e. }z\in K_{i},\quad i=1,\ldots,d,\\
\label{eq-var2}
U^{\mu}_{i}(z) +Q_i& \leq F_{i},\quad \mu_{i}\text{-a.e. } z\in K_{i},\quad i=1,\ldots,d.
\end{align}

\section{Discretization of the vector energy}
Throughout this section, we continue with the same assumptions as above:
\begin{enumerate}
\item $C\geq 0$ and $C$ has no zero columns (or rows); $r_1,\ldots, r_d >0$;
\item $K_1,\ldots,K_d$ nonpolar with $K_i\cap K_j=\emptyset$, $i\neq j$;
\item $Q_1,\ldots,Q_d$ admissible.
\end{enumerate}

To discretize the vector energies $E$ and $E_{Q}$, for each $k=1,2,...$ we take a sequence of ordered tuples  $m_{k}=(m_{1,k},\ldots,m_{d,k})$ of positive integers with 
\begin{equation}\label{correct2}
m_{i,k}\uparrow \infty,\quad i=1,\ldots,d, \quad \hbox{and} \quad \lim_{k\to \infty}  \frac{m_{i,k}}{m_{{j,k}}}=\frac{r_i}{r_j},\quad i,j=1,\ldots,d.
\end{equation} 
Note with this hypothesis
\begin{equation}\label{correct3}\frac{r_i^2}{m_{i,k}^2}\asymp \frac{r_j^2}{m_{j,k}^2} \asymp \frac{r_ir_j}{m_{i,k}m_{j,k}},
\end{equation}
where the notation $a_{k}\asymp b_{k}$ stands for asymptotically equal, i.e. $a_{k}/b_{k}\to 1$ as $k\to\infty$. 
For a set of distinct points of the form
\begin{equation}\label{Zk}
{\bf Z_{k}}=\cup_{i=1}^{d}\{z_{i,1},\ldots,z_{i,m_{i,k}}\in K_{i}\},
\end{equation}
 let
\begin{equation}\label{vdm}
|VDM_k({\bf Z_{k}})|:=\prod_{i=1}^{d}\prod_{l<p}^{m_{i,k}}|z_{i,l}-z_{i,p}|^{c_{i,i}}
\cdot \prod_{i<j}^{d}\prod_{l=1}^{m_{i,k}}\prod_{p=1}^{m_{j,k}}|z_{i,l}-z_{j,p}|^{c_{i,j}}.
\end{equation}
We define a $k-$th order vector diameter with respect to $(m_{1,k},\ldots,m_{d,k})$ -- all that follows will be with respect to a sequence satisfying (\ref{correct2}) --  via
\begin{equation}\label{diam}
\delta^{(k)}(K):=\max_{{\bf Z_{k}}}\Big[|VDM_k({\bf Z_{k}})|\Big]^{2|r|^{2}/|m_{k}|(|m_{k}|-1)},
\end{equation}
where we set
$$|r|=r_{1}+\cdots+r_{d},\qquad |m_{k}|=m_{1,k}+\cdots+m_{d,k}.$$
%^{\frac{n_k^2r_1^2+m_k^2r_2^2}{m_k^2n_k^2}}.$$
Given a weight $Q=(Q_1,\ldots,Q_d)$, we define  
$$|VDM_k^Q({\bf Z_{k}})|:=|VDM_k({\bf Z_{k}})|\cdot 
\prod_{i=1}^{d}\prod_{l=1}^{m_{i,k}} e^{-\frac{m_{i,k}}{r_i}Q_i(z_{i,l})}$$ 
and we have the $k-$th order weighted vector diameter:
\begin{equation}\label{diam-w}
\delta^{(k)}_Q(K):=\max_{{\bf Z_{k}}}\Big[|VDM_k^Q({\bf Z_{k}})|\Big]^{2|r|^{2}/|m_{k}|(|m_{k}|-1)}.
\end{equation}
Note that, similarly to the classical scalar case, the factor $|m_{m}|(|m_{k}|-1)/2$ in the exponent of (\ref{diam}) and (\ref{diam-w}) corresponds to the number of factors in the product (\ref{vdm}).
Actually the $``-1"$ in $|m_{m}|(|m_{k}|-1)/2$ could be dropped but with it the formulas reduce to those in the scalar case.

We start with a general result which will also be used in section 6. The proof is similar to the classical (scalar) case; cf., \cite{ST}.
\begin{proposition}\label{wup}
Take a sequence $\{(m_{1,k},\ldots,m_{d,k})\}$ satisfying (\ref{correct2}) and $\mu=(\mu_1,\dots,\mu_d)\in \mathcal M_r(K)$. Let 
$$\mu^k=(\mu_1^k,\ldots,\mu_d^k):=
(\frac{r_1}{m_{1,k}}\sum_{j=1}^{m_{1,k}} \delta_{z_{1,j}^{(k)}},\ldots,
\frac{r_d}{m_{d,k}}\sum_{j=1}^{m_{d,k}} \delta_{z_{d,j}^{(k)}})$$
be a sequence of discrete measures in $\mathcal M_r(K)$ associated to the array 
%$$\{{\bf Z_k},{\bf W_k}\}:=\{(z_1^{(k)},...,z_{m_k}^{(k)},w_1^{(k)},...,w_{n_k}^{(k)})\}$$ 
$${\bf Z_{k}}=\cup_{i=1}^{d}\{z_{i,1}^{(k)},\ldots,z_{i,m_{i,k}}^{(k)}\in K_{i}\},$$
with $\mu^k\to \mu$ weak-*. Then
\begin{equation}
\label{ineg-mu}
\limsup_{k\to \infty} |VDM_k({\bf Z_k})|^{2|r|^{2}/|m_{k}|(|m_{k}|-1)}\leq e^{-E(\mu)}.
\end{equation}
In the weighted case,
\begin{equation}
\label{ineg-mu-Q}
\limsup_{k\to \infty} |VDM_k^Q({\bf Z_k})|^{2|r|^{2}/|m_{k}|(|m_{k}|-1)}\leq e^{-E_Q(\mu)}.
\end{equation}
\end{proposition}

\begin{proof} We have $\mu_i^k\times \mu_j^k\to \mu_i\times \mu_j$ weak-* for $i,j=1,\ldots,d$. Furthermore, the function $(a,b)\to \log \frac{1}{|a-b|}$ is lowersemicontinuous. 
%By Theorem 1.4 of Chapter 0 of Saff-Totik, 
%$$E^{(r)}(\mu)\leq \liminf_{k\to \infty}\frac {-2}{(r_1+r_2)^2(m_k + n_k)^2}\log |VDM_k({\bf Z_k},{\bf W_k})|$$
%which is the result. 
For a real number $M$ let 
	$$h_M(z,t):=\min [M, \log \frac {1}{|z-t|}]\leq \log \frac {1}{|z-t|}.$$
	Then, for $i=1,\ldots,d$, we have
$$I(\mu_i)=\lim_{M\to \infty}\int_{K_i} \int_{K_i} h_{M}(z,t)d\mu_i(z)d\mu_i(t)$$
$$=\lim_{M\to \infty}\lim_{k \to \infty}\int_{K_i} \int_{K_i} h_{M}(z,t)d\mu_i^k(z)d\mu_i^k(t).$$
Now  
$$h_{M}(z_{i,l}^{(k)},z_{i,p}^{(k)})\leq \log \frac {1}{|z_{i,l}^{(k)}-z_{i,p}^{(k)}|},$$
if $l\not = p$ and hence
$$\int_{K_i}\int_{K_i} h_{M}(z,t)d\mu_i^k(z)d\mu_i^k(t) \leq 
\frac{r_i^2}{m_{i,k}^{2}}\left(m_{i,k}M +\sum_{l\not = p}\log \frac {1}{|z_{i,l}^{(k)}-z_{i,p}^{(k)}|}\right).$$
%Now given $\epsilon >0$, 
Consequently,
\begin{align}\notag
I(\mu_i) & \leq \lim_{M\to \infty}\liminf_{k \to \infty} \frac{r_i^2}{m_{i,k}^{2}}\left({m_{i,k}}M +
\sum_{l\not = p}\log \frac {1}{|z_{i,l}^{(k)}-z_{i,p}^{(k)}|}\right)\\\label{ineg-Ii}
 & =\liminf_{k \to \infty}\frac{r_i^2}{m_{i,k}^{2}}\sum_{l\not = p}\log \frac {1}{|z_{i,l}^{(k)}-z_{i,p}^{(k)}|}.
 \end{align}
%Similarly,
%$$I(\mu_2)\leq \liminf_{k \to \infty}r_2^2\Big((\frac{1}{n_k^2-n_k})\Big[\sum_{i\not = j}\log \frac {1}{|w_i^{(k)}-w_j^{(k)}|}\Big]\Big).$$
Finally, since $K_i\cap K_j =\emptyset$, $i\neq j$, from $\mu_i^k\times \mu_j^k\to \mu_i\times \mu_j$ weak-* we have
\begin{equation}
\label{ineg-Iij}
I(\mu_i,\mu_j)=\lim_{k\to \infty} I(\mu_i^k,\mu_j^k)=\lim_{k\to \infty}\frac{r_ir_j}{m_{i,k}m_{j,k}}\sum_{l=1}^{m_{i,k}}\sum_{p=1}^{m_{j,k}}\log \frac {1}{|z_{i,l}^{(k)}-z_{j,p}^{(k)}|}.
\end{equation}
Putting estimates (\ref{ineg-Ii}) and (\ref{ineg-Iij}) for $i,j=1,\ldots,d$, together gives
$$\limsup_{k \to \infty}\sum_{i=1}^{d}\frac{r_i^2}{m_{i,k}^{2}}\sum_{l\not = p}\log {|z_{i,l}^{(k)}-z_{i,p}^{(k)}|^{c_{i,i}}}+
\lim_{k\to \infty}\sum_{i\neq j}\frac{r_ir_j}{m_{i,k}m_{j,k}}\sum_{l=1}^{m_{i,k}}\sum_{p=1}^{m_{j,k}}\log {|z_{i,l}^{(k)}-z_{j,p}^{(k)}|^{c_{i,j}}}\leq -E(\mu).$$
Then, using (\ref{correct2}) and (\ref{correct3}) leads to
$$\limsup_{k\to \infty}\frac{2|r|^{2}}{|m_{k}|(|m_{k}|-1)}\log |VDM_k({\bf Z_k})|\leq -E(\mu),$$
%Note that, in the previous inequality, we used the fact that the quantities
%$$|VDM_k({\bf Z_{k}})|^{2|r|^{2}/|m_{k}|(|m_{k}|-1)}$$
%remain bounded with $k$,
%which is a consequence of Proposition \ref{eetd}.
which proves (\ref{ineg-mu}).

The weighted case (\ref{ineg-mu-Q}) follows from the unweighted case, (\ref{correct2}) and (\ref{correct3}), and lower semicontinuity of $Q_1,\ldots,Q_d$.
\end{proof}

\begin{proposition} \label{eetd} In the unweighted case,
$$\delta(K):=\lim_{k\to \infty} \delta^{(k)}(K)=e^{-E^*}=e^{-E(\mu^K)}$$
and in the weighted case,
$$\delta_Q(K):=\lim_{k\to \infty} \delta^{(k)}_Q(K)=e^{-E_Q^*}=e^{-E_Q(\mu^{K,Q})}.$$
\end{proposition}
\begin{proof} We prove the unweighted case; the weighted case is similar.  
First observe that if we take any points
%$$({\bf z_k},{\bf w_k})= (z_1,...,z_{m_k},w_1,...,w_{m_k})\in K_1\times K_2,$$
$${\bf Z_{k}}=\cup_{i=1}^{d}\{z_{i,1}^{(k)},\ldots,z_{i,m_{i,k}}^{(k)}\in K_{i}\},$$
then
$$-\frac{|m_k|(|m_{k}|-1)}{2|r|^2} \log \delta^{(k)}(K)\leq-\log |VDM_k
%(z_1,...,z_{m_k},w_1,...,w_{n_k})
({\bf Z_{k}})|$$
$$=\sum_{i=1}^{d}c_{i,i}\sum_{l<p}^{m_{i,k}}\log \frac{1}{|z_{i,l}-z_{i,p}|}
+\sum_{i<j}^{d}c_{i,j} \sum_{l=1}^{m_{i,k}}\sum_{p=1}^{m_{j,k}}\log \frac{1}{|z_{i,l}-z_{j,p}|}.$$
Given any $\sigma=(\sigma_1,\ldots,\sigma_d)=(r_1\overline \sigma_1, \ldots, r_d \overline \sigma_d) \in \mathcal M_r(K)$ where $\overline \sigma_i\in \mathcal M_1(K_i)$, $i=1,\ldots,d$, if we integrate with respect to the probability measure
$$\prod_{i=1}^{d}\prod_{l<p}^{m_{i,k}}d\overline \sigma_i(z_{i,l})d\overline \sigma_i(z_{i,p}) \cdot  
\prod_{i<j}^{d}\prod_{l=1}^{m_{i,k}}\prod_{p=1}^{m_{j,k}}d\overline \sigma_i(z_{i,l})d\overline \sigma_j(z_{j,p})$$
we get
$$-\frac{|m_k|(|m_{k}|-1)}{2|r|^2} \log \delta^{(k)}(K)%\lessapprox 
\leq\sum_{i=1}^{d}c_{i,i}\frac{m_{i,k}(m_{i,k}-1)}{2}I(\overline \sigma_i)
+\sum_{i<j}c_{i,j}m_{i,k}m_{j,k}I(\overline \sigma_i,\overline \sigma_j)$$
%$$\approx \frac{m_k^2}{2r_1^2}c_{1,1}I( \sigma_1)+\frac{n_k^2}{2r_2^2}c_{2,2}I(\sigma_2)+\frac{m_kn_k}{r_1r_2}c_{1,2}I( \sigma_1, \sigma_2).$$
$$=\sum_{i=1}^{d}c_{i,i}\frac{m_{i,k}(m_{i,k}-1)}{2r_{i}^{2}}I(\sigma_i)
+\sum_{i<j}c_{i,j}\frac{m_{i,k}m_{j,k}}{r_{i}r_{j}}I(\sigma_i,\sigma_j).$$
Then we use (\ref{correct3}) to obtain 
\begin{equation}\label{ineq-sigma}
e^{-E(\sigma)}\leq\liminf_{k\to \infty} \Big(\delta^{(k)}(K)\Big).
\end{equation}

Next, let
%$$\{{\bf Z_k},{\bf W_k}\}:=\{(z_1^{(k)},...,z_{m_k}^{(k)},w_1^{(k)},...,w_{m_k}^{(k)})\}$$
$${\bf Z_{k}}=\cup_{i=1}^{d}\{z_{i,1}^{(k)},\ldots,z_{i,m_{i,k}}^{(k)}\in K_{i}\},$$
be a Fekete array of order $k$; i.e., achieving the maximum for $\delta^{(k)}(K)$ in (\ref{diam}). Letting $\mu=(\mu_1,\ldots,\mu_d) \in \mathcal M_r(K)$ be any weak-* limit of the sequence of Fekete measures
$$\mu^k:=\Big(\frac{r_1}{m_{1,k}}\sum_{j=1}^{m_{1,k}} \delta_{z_{1,j}^{(k)}},\ldots,
\frac{r_d}{m_{d,k}}\sum_{j=1}^{m_{d,k}} \delta_{z_{d,j}^{(k)}}\Big),$$
Proposition \ref{wup} gives
$$\limsup_{k\to \infty} \Big[\delta^{(k)}(K)\Big]\leq e^{-E(\mu)}.$$
Thus, with (\ref{ineq-sigma}),
$$E(\mu)\leq\lim_{k\to \infty} \Big[-\log \delta^{(k)}(K)\Big]\leq E(\sigma),$$
for any $\sigma$. Hence the limit exists and equals the energy of any weak-* limit $\mu$ of Fekete measures. Since there exists a unique minimizing measure in $\mathcal M_r(K)$ for $E$, we have $\mu=\mu^{K}$ and $\lim_{k\to \infty}\delta^{(k)}(K)=e^{-E(\mu^{K})}$.
\end{proof}

Note that our definition of the $k-$th order (weighted) diameter is relative to $m_k$, but the proof shows that the (weighted) transfinite diameter $\delta(K)$ ($\delta_Q(K)$) is independent of the sequence $m_k$ satisfying (\ref{correct2}). %The same remark applies to the results below.}

 The proof of Proposition \ref{eetd} included the result that (weighted) Fekete measures $\mu^k$ converge weak-* to the (weighted) energy minimizing measure $\mu^K$ ($\mu^{K,Q}$). Indeed, the proof shows the result for {\it asymptotic} (weighted) Fekete measures:
\begin{proposition}\label{awf} In the unweighted case, for an array 
$${\bf Z_{k}}=\cup_{i=1}^{d}\{z_{i,1}^{(k)},\ldots,z_{i,m_{i,k}}^{(k)}\in K_{i}\},$$
%$$\{{\bf Z_k},{\bf W_k}\}:=\{(z_1^{(k)},...,z_{m_k}^{(k)},w_1^{(k)},...,w_{n_k}^{(k)})\},$$
 if
$$\lim_{k\to \infty} |VDM_k({\bf Z_k})|^{{2|r|^{2}}/{|m_k|(|m_{k}|-1)}}=e^{-E^*}$$
then 
$$\mu^k:=\Big(\frac{r_1}{m_{1,k}}\sum_{j=1}^{m_{1,k}} \delta_{z_{1,j}^{(k)}},\ldots,
\frac{r_d}{m_{d,k}}\sum_{j=1}^{m_{d,k}} \delta_{z_{d,j}^{(k)}}\Big)
\to \mu^K \ \hbox{weak}-*$$
and in the weighted case, if
$$\lim_{k\to \infty} |VDM_k^Q({\bf Z_k})|^{{2|r|^{2}}/{|m_k|(|m_{k}|-1)}}=e^{-E_Q^*}$$
then 
$$\mu^k:=\Big(\frac{r_1}{m_{1,k}}\sum_{j=1}^{m_{1,k}} \delta_{z_{1,j}^{(k)}},\ldots,
\frac{r_d}{m_{d,k}}\sum_{j=1}^{m_{d,k}} \delta_{z_{d,j}^{(k)}}\Big)
\to \mu^{K,Q} \ \hbox{weak}-*.$$

\end{proposition}
\begin{proof} We prove the unweighted case; the weighted case is similar. Let $\sigma=(\sigma_1,\ldots,\sigma_d) \in \mathcal M_r(K)$ be any weak-* limit of the sequence of measures $\mu^k$. The proof of Proposition \ref{eetd} shows that 
$$E^*=\limsup_{k\to \infty} \Big(-\log \delta^{(k)}(K)\Big)\leq E(\sigma);$$
then Proposition \ref{wup} gives 
$$E(\sigma)\leq \liminf_{k\to \infty} \Big[{\frac{-2|r|^{2}}{|m_k|(|m_{k}|-1)}}
\log|VDM_k({\bf Z_{k}})|\Big]=E^*.$$
Thus 
$$E^*=E(\sigma)$$
so that $\sigma$ minimizes $E$ over all $\mu \in \mathcal M_r(K)$. Since there  exists a unique minimizer for $E$, we are done.

\end{proof}

\noindent Again, if we need to keep track of the underlying interaction matrix $C$, we write
$$-\log \delta^C(K)=(E^C)^* \ \hbox{and} \ -\log \delta_Q^C(K)=(E^C_Q)^*.$$
Occasionally we may write $VDM_k^{(C)}$ as well. If $\alpha \in \C\setminus \{0\}$, then $K_i\cap K_j=\emptyset$ implies $\alpha K_i\cap \alpha K_j=\emptyset$ and we have, using the definitions of $\delta^{(k)}(K)$ and $\delta^{(k)}_Q(K)$ together with (\ref{correct2}) and (\ref{correct3}), the scaling relations 
\begin{equation}\label{scal}\delta^C(\alpha K) = |\alpha|^{B}\delta^C(K) \ \hbox{and} \ \delta_Q^C(\alpha K) = |\alpha|^{B}\delta_Q^C(K)\end{equation}
where 
$$B=B(C,r)= \sum_{i,j=1}^dc_{i,j}r_ir_j\geq 0.$$
Note that $B$ is independent of the sequence $m_k$ used to define the $k-$th order diameters.

We use Proposition \ref{eetd} and (\ref{scal}) to prove an important continuity property of the (weighted) vector transfinite diameter.

\begin{proposition}\label{tdlim} Given $C:=(c_{i,j})_{i,j=1}^{d}$
%\begin{pmatrix}
%c_{1,1} & c_{1,2}\\
%c_{2,1} & c_{2,2}
%\end{pmatrix}$, 
we can find $C^{(k)}:=(c_{i,j}^{(k)})_{i,j=1}^{d}$
%\begin{pmatrix}
%c^j_{1,1} & c^j_{1,2}\\
%c^j_{2,1} & c^j_{2,2}
%\end{pmatrix}$ 
symmetric positive semidefinite with all entries $c^{(k)}_{i,j}$ rational, $C^{(k)}\to C$ componentwise, and 
$$\lim_{k\to \infty}\delta^{C^{(k)}}(K)=\delta^C(K)\ \hbox{and} \ \lim_{k\to \infty}\delta_Q^{C^{(k)}}(K)=\delta_Q^C(K).$$
\end{proposition}
 \begin{proof} From Proposition \ref{eetd}, we can instead work with the (weighted) minimal energies. We first prove the unweighted case. 
 We take $c^{(k)}_{i,j}$ rational with $c^{(k)}_{i,j}\downarrow c_{i,j}$ for $c_{i,j}\geq 0$ and $c^{(k)}_{i,j}\uparrow c_{i,j}$ for $c_{i,j}<0$. Note that, by choosing 
 $|c_{i,j}-c_{i,j}^{(k)}|$, $i\neq j$, sufficiently small with respect to $c_{i,i}^{(k)}-c_{i,i}$, $i=1,\ldots,d$, the matrix $C^{(k)}$ is symmetric positive semidefinite.
% First take all $c_{ij}\geq 0$. We take $c^j_{kl}$ rational with $c^j_{k,l}\downarrow c_{k,l}$ for $k,l=1,2$ {\red and $C_j$ symmetric positive semidefinite}. 

Let $\mu^K=(\mu_1,\ldots,\mu_d)$ satisfy $E^C(\mu^K) =(E^C)^*$. By rescaling (see (\ref{scal})), we may assume $K_1,\ldots,K_d$ are contained in a disk of radius $1/2$ so that all energies $I(\mu_i)$ and $I(\mu_i,\mu_j)$ are nonnegative. Then
 $$(E^C)^*\leq (E^{C^{(k)}})^*\leq E^{C^{(k)}}(\mu^K).$$
% Fix $M>0$ with 
% $$I(\mu_1),I(\mu_2),I(\mu_1,\mu_2) \leq M.$$
% Given $\epsilon >0$, let $\delta = \epsilon/M$. Then if $c^j_{k,l} < c_{k,l}+\delta=c_{k,l}+\epsilon/M$,
% $$E^{C_j}(\mu^K)=c^j_{1,1}I(\mu_1)+c^j_{2,2}I(\mu_2)+2c^j_{1,2}I(\mu_1,\mu_2)$$
% $$\leq E^C(\mu^K)+4\epsilon.$$
 
% If $c_{1,2} <0$, we take $c^j_{kl}$ rational with $c^j_{k,k}\downarrow c_{k,k}$ for $k=1,2$ and $c^j_{1,2}\uparrow c_{1,2}$ {\red (note then $C_j$ is symmetric positive semidefinite)}. Then if $\mu^K_C=(\mu_1,\mu_2)$ satisfy $E^C(\mu^K) =(E^C)^*$, we again have
% $$(E^C)^*\leq (E^{C_j})^*\leq E^{C_j}(\mu^K).$$
 Now, simply by continuity, since $C^{(k)}\to C$, given $\epsilon >0$, 
 $$|E^{C^{(k)}}(\mu^K)-E^C(\mu^K)| <\epsilon$$ 
 for $k$ sufficiently large and the result follows.
 
 For the weighted case, let $\mu^{K,Q}$ satisfy $E_Q^C(\mu^{K,Q}) =(E_Q^C)^*$. Again from (\ref{scal}) we can assume all $K_i$ are contained in a disk of radius $1/2$ and we have the similar inequality
 $$(E_Q^C)^*\leq (E_Q^{C^{(k)}})^*\leq E_Q^{C^{(k)}}(\mu^{K,Q}).$$
 The proof proceeds as in the unweighted case. 
  \end{proof}

\section{Bernstein-Markov properties}
In the first subsection, we define the notion of {\it strong rational Bernstein-Markov property} and we show that on any nonpolar compact set of $\C$ there exists a positive measure that satisfies such a property.
In the second subsection, we define a vector analog of this notion and we use it to show that the $L^{2}$ versions of the $k$-th order vector diameters defined in (\ref{diam}) and (\ref{diam-w}) have the same asymptotic behavior as $k$ tends to infinity.
\subsection{Bernstein-Markov properties in $\C^{n}$}
For any $n=1,2,...$, let $\mathcal P_k=\mathcal P_k^{(n)}$ denote the holomorphic polynomials  in $n$ variables of degree at most $k$. Given a compact set $K\subset \C^n$ and a measure $\nu$ on $K$, we say that $(K,\nu)$ satisfies a Bernstein-Markov property if for all $p_k\in \mathcal P_k$, 	
$$||p_k||_K:=\sup_{z\in K} |p_k(z)|\leq  M_k||p_k||_{L^2(\nu)}  \ \hbox{with} \ \limsup_{k\to \infty} M_k^{1/k} =1.$$

We will need to use the Bernstein-Markov property in $\C^2$ to derive properties in the univariate case. It was shown in \cite{PELD} that any compact set in $\C^n$ admits a Bernstein-Markov measure; indeed, the following stronger statement is true.

\begin{proposition} [\cite{PELD}]  \label{allbm} Let $K\subset \R^n$. There exists a measure $\nu \in \mathcal M(K)$ such that for all complex-valued polynomials $p$ of degree at most $k$ in the (real) coordinates $x=(x_1,...,x_n)$ we have
$$||p||_K\leq M_k ||p||_{L^2(\nu)}$$
where $\limsup_{k\to \infty}M_k^{1/k}=1$.

\end{proposition} 

More generally, for $K\subset \C^n$ compact, $Q$ admissible ($Q$ is lowersemicontinuous and finite on a nonpluripolar set), and $\nu$ a measure on $K$, we say that the triple $(K,\nu,Q)$ satisfies a weighted Bernstein-Markov property if for all $p_k\in \mathcal P_k$, 
$$||e^{-kQ}p_k||_K \leq M_k ||e^{-kQ}p_k||_{L^2(\nu)} \ \hbox{with} \ \limsup_{k\to \infty} M_k^{1/k} =1.$$ 
Here $K$ should be nonpluripolar for this notion to have any content. For the definition of pluripolar, the $\C^n-$analogue of polar, see Appendix B of \cite{ST}.

\begin{remark} \label{tomobs} An important observation is the following. If $(K,\nu,Q)$ satisfies a weighted Bernstein-Markov property for some admissible weight $Q$ on $K$, then for any sequence $\{Q_k\}$ of admissible weights on $K$ which converges uniformly to $Q$ on $K$, we have a ``varying weight'' Bernstein-Markov property:
\begin{equation}\label{varwt}\lim_{k\to \infty} \Big( \sup_{p_k\in \mathcal P_k}\frac{||e^{-kQ_k}p_k||_K}{||e^{-kQ_k}p_k||_{L^2(\mu)}}\Big)^{1/k}=1.\end{equation}
To verify (\ref{varwt}), note simply that given $\epsilon >0$ we have $$e^{-kQ}e^{-k\epsilon}<e^{-kQ_k}< e^{-kQ}e^{k\epsilon} $$ on all of $K$ for $k$ sufficiently large.

\end{remark}

These properties can be stated using $L^p(\nu)$ in place of $L^2(\nu)$, but it is known that {\it if $(K,\nu)$ satisfies an (weighted) $L^p-$Bernstein-Markov property for some $0<p<\infty$ then $(K,\nu)$ satisfies an (weighted) $L^p-$Bernstein-Markov property for all $0<p<\infty$}. This follows, for example, from Remark 3.2 in \cite{bloomIU}; see also the proof of Theorem 3.4.3 in \cite{StTo}. Thus, we simply say that $(K,\nu)$ satisfies a {\it (weighted) Bernstein-Markov property}.

\begin{definition} We say $(K,\nu)$ satisfies a {\it strong Bernstein-Markov property} if $(K,\nu,Q)$ satisfies a weighted Bernstein-Markov property for each {\it continuous} $Q$. 
\end{definition}

\noindent Again, $K$ should be nonpluripolar for this notion to have any content. 

Now we return to $n=1$; i.e., $\C$, and we next give a definition of a ``rational'' (weighted) Bernstein-Markov property, analogous to the definition for polynomials and for which the proof that this property being valid for some $p>0$ implies it is valid for all $p>0$ remains true. The paper \cite{BEZ} also concerns a rational Bernstein-Markov property. Given $K\subset \C$ compact, we fix a compact set $K'$ disjoint from $K$ and define, for $a,b>0$,
\begin{equation}\label{rnclass}{\mathcal R}_k=\{r_k=p_k/q_k: p_k, q_k \ \hbox{polynomials}; \ \deg p_k \leq ak, \ \deg q_k \leq bk; \ \hbox{all zeros of} \ q_k \ \hbox{in} \ K'\}.\end{equation}
We say that $(K,\nu)$ satisfies a {\it rational} Bernstein-Markov property if for all $r_k\in \mathcal R_k$, 	
$$||r_k||_K:=\sup_{z\in K} |r_k(z)|\leq  M_k||r_k||_{L^2(\nu)}  \ \hbox{with} \ \limsup_{k\to \infty} M_k^{1/k} =1.$$

\noindent Here ${\mathcal R}_k={\mathcal R}_k(K',a,b)$. Note that taking $q_k\equiv 1$ we see that $(K,\nu)$ satisfies a (polynomial)  Bernstein-Markov property. 

More generally, for $K\subset \C$ compact, $Q$ admissible, and $\nu$ a measure on $K$, we say that the triple $(K,\nu,Q)$ satisfies a weighted rational Bernstein-Markov property if for all $r_k\in \mathcal R_k$, 
$$||e^{-kQ}r_k||_K \leq M_k ||e^{-kQ}r_k||_{L^2(\nu)} \ \hbox{with} \ \limsup_{k\to \infty} M_k^{1/k} =1.$$

\begin{definition} We say $(K,\nu)$ satisfies a {\it strong rational Bernstein-Markov property} if $(K,\nu,Q)$ satisfies a weighted rational Bernstein-Markov property for each {\it continuous} $Q$. 
\end{definition}

In the definitions of these various rational Bernstein-Markov properties, there is an implicit underlying pole set $K'$ as well as positive numbers $a,b$. We will specify $K',a,b$ in our vector setting in subsection 4.2.  

To define certain vector energy functionals in section 6, and for our large deviation principle in section 7, we will need to use measures satisfying vector versions of the strong (rational) Bernstein-Markov property. We next prove that such measures always exist on nonpolar compacta in the scalar case; it will be clear from the proof that the constructed measures work for any fixed pole set $K'$ and positive numbers $a,b$. For simplicity we take $b=1$.

\begin{theorem} \label{strongrbm} Let $K\subset \C$ be nonpolar. Then there exists $\nu$ on $K$ with $(K,\nu)$ satisfying a strong rational Bernstein-Markov property.
\end{theorem}

\begin{proof} We consider $K\subset \C =\R^2\subset \C^2$ with variables $(z_1,z_2)$ where $\Re z_1 =x$ and $\Re z_2 =y$ so that $z=x+iy$ is the usual complex variable when we consider $\C =\R^2$. Using Proposition \ref{allbm}, we construct a measure $\nu$ on $K$ such that $(K,\nu)$ satisfies a Bernstein-Markov property with respect to holomorphic polynomials on $\C^{2}$. Theorem 3.2 of \cite{bloom} then shows that $(K,\nu,Q)$ satisfies a weighted Bernstein-Markov property with respect to holomorphic polynomials on $\C^{2}$ for all $Q\in C(K)$; i.e., $(K,\nu)$ satisfies a strong Bernstein-Markov property with respect to holomorphic polynomials on $\C^{2}$. Since a holomorphic polynomial in $z$ of degree at most $n$ is of the form 
$$p_n(z) =\sum a_jz^j =\sum a_j (x+iy)^j =\sum c_{kl} x^ky^l =\sum c_{kl} (\Re z_1)^k(\Re z_2)^l$$
where $c_{kl}$ are complex numbers and $k+l\leq n$, each such $p_n$ is the restriction to $\R^2$ of a holomorphic polynomial $\tilde p_n(z_1,z_2):=\sum c_{kl}z_1^k z_2^l$ in $\C^2$. Thus $(K,\nu)$ satisfies a strong Bernstein-Markov property with respect to holomorphic polynomials on $\C$.

Applying Remark \ref{tomobs}, $(K,\nu)$ satisfies a ``varying weight'' Bernstein-Markov property for any continuous target weight: for any $Q\in C(K)$, and any sequence $\{Q_k\}$ of admissible weights on $K$ which converges uniformly to $Q$ on $K$, (\ref{varwt}) holds:
$$\lim_{k\to \infty} \Big( \sup_{p_k\in \mathcal P_k}\frac{||e^{-kQ_k}p_k||_K}{||e^{-kQ_k}p_k||_{L^2(\mu)}}\Big)^{1/k}=1.$$

We now fix $Q\in C(K)$ and consider the sequence of numbers
$$\{\Big( \sup_{r_n\in \mathcal R_n}\frac{||e^{-nQ}r_n||_K}{||e^{-nQ}r_n||_{L^2(\mu)}}\Big)^{1/n}\}.$$
Let
$$\alpha:=\limsup_{n\to \infty} \Big( \sup_{r_n\in \mathcal R_n}\frac{||e^{-nQ}r_n||_K}{||e^{-nQ}r_n||_{L^2(\mu)}}\Big)^{1/n}.$$
Clearly $\alpha \geq 1$; we want to show $\alpha = 1$. Take a subsequence $\{n_k\}$ of integers so that 
$$\lim_{k\to \infty} \Big( \sup_{r_{n_k}\in \mathcal R_{n_k}}\frac{||e^{-n_kQ}r_{n_k}||_K}{||e^{-n_kQ}r_{n_k}||_{L^2(\mu)}}\Big)^{1/n_k}=\alpha$$
and, given $\epsilon >0$, choose $r_{n_k}\in \mathcal R_{n_k}$ with 
$$\Big(\frac{||e^{-n_kQ}r_{n_k}||_K}{||e^{-n_kQ}r_{n_k}||_{L^2(\mu)}}\Big)^{1/n_k} \geq \Big( \sup_{r_{n_k}\in \mathcal R_{n_k}}\frac{||e^{-n_kQ}r_{n_k}||_K}{||e^{-n_kQ}r_{n_k}||_{L^2(\mu)}}\Big)^{1/n_k}-\epsilon.$$
Writing $r_{n_k}:=p_{n_k}/q_{n_k}$ where we take $q_{n_k}=\prod_{j=1}^{n_k}(z-z_j^{(k)})$ monic with zeros in $K'$, we have
$$e^{-n_kQ}|r_{n_k}|=\frac{e^{-n_kQ}}{|q_{n_k}|}\cdot |p_{n_k}|=: e^{-n_kQ_{n_k}} \cdot |p_{n_k}|$$
where
$$e^{-n_kQ_{n_k}}= \frac{e^{-n_kQ}}{|q_{n_k}|} \ \hbox{so that} \ Q_{n_k}= Q+ \frac{1}{n_k}\log |q_{n_k}|.$$
Now $({-1}/{n_k})\log |q_{n_k}|$ is the logarithmic potential $U^{\mu_k}$ of the probability measure 
$$\mu_k:=\frac{1}{n_k}\sum_{j=1}^{n_k}\delta_{z_j^{(k)}}$$ which is supported in $K'$. Taking a weak-* limit of this sequence $\{\mu_k\}$ we get a probability measure $\nu$ on $K'$ with 
$U^{\mu_k}\to U^{\nu}$ uniformly on $K$; hence, taking the corresponding subsequence of $\{n_k\}$ (which we do not relabel) we have
$$Q_{n_k} \to  Q-  U^{\nu}\ \hbox{uniformly on} \ K.$$
Note that $U^{\nu}$ is harmonic and hence continuous on $K$. 
We extend the definition of $Q_n$ for $n\notin \{n_k\}$ by simply defining $Q_n:= Q-  U^{\nu}$ for such $n$. Then the full sequence $\{Q_n\}$ satisfies $Q_n \to Q-U^{\nu}$ uniformly on $K$ and thus we have from (\ref{varwt}) that $$\lim_{n\to \infty} \Big( \sup_{p_n\in \mathcal P_n}\frac{||e^{-nQ_n}p_n||_K}{||e^{-nQ_n}p_n||_{L^2(\mu)}}\Big)^{1/n}=1.$$
But for $n=n_k$ we have 
$$\Big( \sup_{p_{n_k}\in \mathcal P_{n_k}}\frac{||e^{-n_kQ_{n_k}}p_{n_k}||_K}{||e^{-n_kQ_{n_k}}p_{n_k}||_{L^2(\mu)}}\Big)^{1/n_k}  \geq \Big(\frac{||e^{-n_kQ}r_{n_k}||_K}{||e^{-n_kQ}r_{n_k}||_{L^2(\mu)}}\Big)^{1/n_k}$$
$$\geq \Big( \sup_{r_{n_k}\in \mathcal R_{n_k}}\frac{||e^{-n_kQ}r_{n_k}||_K}{||e^{-n_kQ}r_{n_k}||_{L^2(\mu)}}\Big)^{1/n_k}-\epsilon.$$
Thus
$$\alpha =\lim_{k\to \infty} \Big( \sup_{r_{n_k}\in \mathcal R_{n_k}}\frac{||e^{-n_kQ}r_{n_k}||_K}{||e^{-n_kQ}r_{n_k}||_{L^2(\mu)}}\Big)^{1/n_k}=1.$$

\end{proof}

\begin{remark} There are easy-to-check sufficient conditions for a measure to satisfy a
strong (rational) Bernstein-Markov property. Let $K\subset \C^n$. We say $(K,\nu)$ satisfies a {\it mass-density property} if there exists $T>0$ with $\nu(B(z_0,r))\geq r^T$ for all $z_0\in K$ and all $r<r(z_0)$ where $B(z_0,r)$ is the ball of radius $r$ centered at $z_0$. For $K$ regular in the pluripotential-theoretic sense (see Appendix B of \cite{ST}), this property implies that $(K,\nu)$ satisfies a Bernstein-Markov property; hence if $K\subset \R^2\subset \C^2$ has this regularity and $(K,\nu)$ satisfies a mass-density property, then the proof of Theorem \ref{strongrbm} shows that $(K,\nu)$ satisfies a strong rational Bernstein-Markov property. In particular, if $K=\bar D$ when $D$ is a bounded domain in $\R^2$ with $C^1-$boundary, any $\nu$ which is a positive, continuous multiple of Lebesgue measure on $D$ is a 
strong rational Bernstein-Markov measure for $K$.
\end{remark}

\subsection{Vector Bernstein-Markov property}

\begin{definition} \label{bmprop} Let $0<p<\infty$ and let $\nu=(\nu_1,\ldots,\nu_d)$ be a tuple of measures with $\nu_i$ supported in $K_i$ for $i=1,\ldots,d$. Recall $K=(K_1,...,K_d)$. We say $(K,\nu)$ satisfies an {\it $L^p-$Bernstein-Markov property} if for $i=1,\ldots,d$,
$$||p_k||_{K_i}\leq M^{(p)}_{k,i} ||p_k||_{L^p(\nu_i)}, \ p_k\in \mathcal P_k$$
where $(M^{(p)}_{k,i})^{1/k}\to 1$ as $k\to\infty$; i.e., each $(K_i,\nu_i)$ satisfies an $L^p-$ Bernstein-Markov property. 
\end{definition}

It follows from the scalar case that {\it if $(K,\nu)$ satisfies an $L^p-$Bernstein-Markov property for some $0<p<\infty$ then $(K,\nu)$ satisfies an $L^p-$Bernstein-Markov property for all $0<p<\infty$}. Thus, we simply say, in our vector setting, that $(K,\nu)$ satisfies a {\it Bernstein-Markov property}.

Now let $Q=(Q_1,\ldots,Q_d)$ be a $d-$tuple of admissible weights for $K=(K_1,...,K_d)$.
\begin{definition} We say $(K,\nu,Q)$ satisfies an {\it $L^p-$weighted Bernstein-Markov property} if  for $i=1,\ldots,d$,
$$||p_ke^{-kQ_i}||_{K_i}\leq M^{(p)}_{k,i} ||p_ke^{-kQ_i}||_{L^p(\nu_i)}, \ p_k\in \mathcal P_k$$
where $(M^{(p)}_{k,i})^{1/k}\to 1$ as $k\to\infty$; i.e., each $(K_i,\nu_i,Q_i)$ satisfies an $L^p-$ weighted Bernstein-Markov property. 
\end{definition}

\begin{definition} We say $(K,\nu)$ satisfies a {\it strong Bernstein-Markov property} if $(K,\nu,Q)$ satisfies a weighted Bernstein-Markov property for each {\it continuous} $Q$. 
\end{definition}

We next define vector versions of rational Bernstein-Markov properties. Our setting is the following: the classes $\mathcal R_k$ defined in (\ref{rnclass}) will be taken with $K=K_{i}$ and $K'=\cup_{j\neq i}K_{j}$, for $i=1,\ldots,d$:
$${\mathcal R}^i_k=\{r_k=p_k/q_k: p_k, q_k \ \hbox{polynomials}; \ \deg p_k \leq ak, \ \deg q_k \leq bk; \ \hbox{all zeros of} \ q_k \ \hbox{in} \ \cup_{j\neq i}K_j\}.$$
Given an interaction matrix $C\geq 0$ and $r_1,\ldots, r_d >0$, the $a,b$ we choose will depend on the coefficients $c_{i,j}$ of $C$ as well as $r_1,\ldots, r_d$.

\begin{definition} \label{rbmprop} Let $0<p<\infty$ and let $\nu=(\nu_1,\ldots,\nu_d)$ be a tuple of measures with $\nu_i$ supported in $K_i$ for $i=1,\ldots,d$. We say $(K,\nu)$ satisfies an {\it $L^p-$rational Bernstein-Markov property} if for $i=1,\ldots,d$,
$$||r_k||_{K_i}\leq M^{(p)}_{k,i} ||r_k||_{L^p(\nu_i)}, \ r_k\in \mathcal R^i_k$$
where $(M^{(p)}_{k,i})^{1/k}\to 1$ as $k\to\infty$; i.e., each $(K_i,\nu_i)$ satisfies an $L^p-$rational Bernstein-Markov property. 
\end{definition}

From the scalar setting again we simply say that $(K,\nu)$ satisfies a {\it rational  Bernstein-Markov property} since the property holds for all $p>0$ once it holds for any $p>0$. Also, as in the scalar case, if $(K,\nu)$ satisfies a rational  Bernstein-Markov property then $(K,\nu)$ satisfies a (polynomial) Bernstein-Markov property.

\begin{definition} For $Q=(Q_1,...,Q_d)$, we say $(K,\nu,Q)$ satisfies an {\it $L^p-$weighted rational Bernstein-Markov property} if  for $i=1,\ldots,d$,
$$||r_ke^{-kQ_i}||_{K_i}\leq M^{(p)}_{k,i} ||r_ke^{-kQ_i}||_{L^p(\nu_i)}, \ r_k\in \mathcal R^i_k$$
where $(M^{(p)}_{k,i})^{1/k}\to 1$ as $k\to\infty$; i.e., each $(K_i,\nu_i,Q_i)$ satisfies an $L^p-$ weighted rational Bernstein-Markov property. 
\end{definition}

\begin{definition} We say $(K,\nu)$ satisfies a {\it strong rational Bernstein-Markov property} if $(K,\nu,Q)$ satisfies a weighted rational Bernstein-Markov property for each {\it continuous} $Q$. 
\end{definition}

Appealing to the scalar case result that any nonpolar compact set $K \subset \C$ admits a measure $\mu$ such that $(K,\mu)$ satisfies a strong rational Bernstein-Markov property (Theorem \ref{strongrbm}), we thus have the analogous result in the vector case: any nonpolar tuple $K=(K_1,\ldots,K_d)$ admits a strong rational Bernstein-Markov tuple $\nu=(\nu_1,\ldots,\nu_d)$.

\begin{remark} \label{caution} First a word on notation: given a sequence $\{m_k\}$ satisfying (\ref{correct2}) and a sequence $\{{\bf Z_{k}}\}$ of points of the form (\ref{Zk}), we write, with abuse of notation, $K^{k}:=K_1^{m_{1,k}}\times\ldots\times K_d^{m_{d,k}}$, and 
$$d\nu({\bf Z_k}):=d\nu_{1}(z_{1,1})\ldots d\nu_{1}(z_{1,m_{1,k}}) d\nu_{2}(z_{2,1})\ldots d\nu_{d}(z_{d,m_{d,k}}).$$ 
Next, given $C\geq 0$ and $r_1,\ldots, r_d >0$ in our vector energy setting, when we write ``Bernstein-Markov property'' below -- and essentially for the rest of the paper -- we will mean ``polynomial Bernstein-Markov property'' if all coefficients $c_{i,j}$ of $C$ are nonnegative and ``rational Bernstein-Markov property'' otherwise.
\end{remark}
\begin{proposition} \label{weightedtd} Let $\{m_k\}$ be a sequence satisfying (\ref{correct2}) and ${\bf Z_{k}}$ a set of points of the form (\ref{Zk}). Assume $(K,\nu)$ satisfies a Bernstein-Markov property. Let
$$Z_k:= \int_{K^{k}}|VDM_k({\bf Z_k})|^2d\nu({\bf Z_k}).$$
Then
$$\lim_{k\to \infty} Z_k^{|r|^2/|m_k|(|m_k|-1)}=e^{-E^*}=\delta^C(K).$$
In the weighted case, if $(K,\nu,Q)$ satisfies a weighted Bernstein-Markov property and 
$$Z^Q_k:= \int_{K^{k}}|VDM_k^Q({\bf Z_k})|^2d\nu({\bf Z_k}),$$
then
$$\lim_{k\to \infty} (Z_k^Q)^{|r|^2/|m_k|(|m_k|-1)}=e^{-E_Q^*}=\delta_Q^C(K).$$
\end{proposition}

\begin{proof} We prove the unweighted version; the weighted version is similar. Clearly 
$$Z_k^{|r|^2/|m_k|(|m_k|-1)}\leq \delta^{(k)}(K)[\nu(K^k)]^{|r|^2/|m_k|(|m_k|-1)},$$
and by letting $k\to\infty$,
$$\limsup_{k\to\infty}Z_k^{|r|^2/|m_k|(|m_k|-1)}\leq \delta^{C}(K).$$
Recall that 
$$|VDM_k({\bf Z_{k}})|:=\prod_{i=1}^{d}\prod_{l<p}^{m_{i,k}}|z_{i,l}-z_{i,p}|^{c_{i,i}}
\cdot\prod_{i<j}^{d} \prod_{l=1}^{m_{i,k}}\prod_{p=1}^{m_{j,k}}|z_{i,l}-z_{j,p}|^{c_{i,j}}.$$
\vskip4pt
\noindent {\bf Case I: All coefficients $c_{i,j}$ are integers.}
\vskip4pt
%First assume each $c_{i,j}$ is a nonnegative integer. Then 
It is easily checked that $VDM_k({\bf Z_{k}})$ is a rational function whose numerator and denominator degrees are bounded by
$$\max_{i}\left(\sum_{j=1}^{d}m_{j,k}|c_{i,j}|\right)\leq A|m_k|$$ 
in each variable where $A=A(C)=\max(|c_{i,j}|)$. 
%{\blue Note that $$\prod_{i<j}^{m_k}|z_i-z_j|\cdot \prod_{i<j}^{n_k}|w_i-w_j|\cdot \prod_{i=1}^{m_k}\prod_{j=1}^{n_k}|z_i-w_j|$$
%is the modulus of a polynomial of degree at most $m_k+n_k$ in each variable, but the quantity $VDM_k(z_1,...,z_{m_k},w_1,...,w_{n_k})$, as a function, e.g., of $z_1$, has different degrees/powers.}

Let ${\bf A_{k}}=(a_{1,1},...,a_{d,m_{d,k}})$ be a set of Fekete points of order $k$ for $K$. Then 
$$p(z_{1,1}):=VDM_k(z_{1,1},a_{1,2},...,a_{d,m_{d,k}})$$
is a rational function in $z_1$ of numerator and denominator degrees at most $A|m_k|$ achieving its supremum norm on $K_1$ at $z_{1,1}=a_{1,1}$. By the Bernstein-Markov property, we have
$$|VDM_k({\bf A_{k}})|^2\leq M_{|m_k|}^2\int_{K_1}|VDM_k(z_{1,1},a_{1,2},...,a_{d,m_{d,k}})|^2d\nu_1(z_{1,1}).$$
Now for each fixed $z_{1,1}\in K_1$, we consider 
$$q(z_{1,2}):=VDM_k(z_{1,1},z_{1,2},...,a_{d,m_{d,k}})$$
as a rational function in $z_{1,2}$ of numerator and denominator degrees at most 
$A|m_k|$. Again, by the Bernstein-Markov property, we have
$$|q(a_{1,2})^{2}|\leq ||q||^{2}_{K_1} \leq M_{|m_k|}^{2} \int_{K_1} |q(z_{1,2})|^2d\nu_1(z_{1,2}).$$ 
Inserting this in the integrand of our previous estimate gives
\begin{equation*}
|VDM_k({\bf A_{k}})|^2 
\leq M_{|m_k|}^4\int_{K_1}\int_{K_1}|VDM_k(z_{1,1},z_{1,2},...,a_{d,m_{d,k}})|^2d\nu_1(z_{1,1})d\nu_1(z_{1,2}).
\end{equation*}
Continuing in this way, we obtain
$$|VDM_k({\bf A_{k}})|^2\leq M_{|m_k|}^{2|m_k|}Z_k.$$
This says that
$$\delta^{(k)}(K)\leq 
M_{|m_k|}^{2|r|^{2}/(|m_{k}|-1)}
Z_k^{|r|^{2}/|m_{k}|(|m_{k}|-1)}$$
and we are done since 
%$m_k + n_k=0(m_k)$ and hence $2(m_k + n_k){\frac{n_k^2r_1^2+m_k^2r_2^2}{m_k^2n_k^2}}=0(m_k^{-2})$.
$M_{|m_{k}|}^{1/(|m_{k}|-1)}\to 1$ as $k\to\infty$.
\vskip4pt
\noindent {\bf Case II: All coefficients $c_{i,j}$ are rational numbers.}
\vskip4pt
Let $M$ be a positive integer such that each $Mc_{i,j}$ is an integer. Now  
$$p(z_{1,1}):=VDM_k(z_{1,1},a_{1,2},...,a_{d,m_{d,k}})^M$$
is a rational function in $z_{1,1}$ of numerator and denominator degrees at most $AM|m_k|$ achieving its supremum norm on $K_1$ at $z_{1,1}=a_{1,1}$. Applying the $L^p-$Bernstein-Markov property to this rational function with exponent $p=2/M$
we have
$$|VDM_k({\bf A_{k}})|^2=\Big(|VDM_k({\bf A_{k}})|^M\Big)^{2/M}$$
$$\leq \Big(M^{(2/M)}_{|m_k|}\Big)^{2/M}\int_{K_1}|VDM_k(z_{1,1},a_{1,2},...,a_{d,m_{d,k}})|^2d\nu_1(z_1).$$
For each fixed $z_{1,1}\in K_1$, we now consider 
$$q(z_{1,2}):=VDM_k(z_{1,1},z_{1,2},...,a_{d,m_{d,k}})^M$$
as a rational function in $z_{1,2}$ of degree at most $AM|m_k|$. We have
\begin{multline*}
|VDM_k(z_{1,1},a_{1,2},...,a_{d,m_{d,k}})|^2=|q(a_{1,2})|^{2/M}\leq ||q||_{K_1}^{2/M} \\
\leq \Big(M^{(2/M)}_{|m_k|}\Big)^{2/M} \int_{K_1} |q(z_{1,2})|^{2/M}d\nu_1(z_{1,2}) \\
=\Big(M^{(2/M)}_{|m_k|}\Big)^{2/M} \int_{K_1} |VDM_k(z_{1,1},z_{1,2},...,a_{d,m_{d,k}})|^2d\nu_1(z_{1,2}).
\end{multline*}
Inserting this in the integrand of our previous estimate gives
\begin{equation*}
|VDM_k({\bf A_{k}})|^2 
\leq \Big(M^{(2/M)}_{|m_k|}\Big)^{4/M}\int_{K_1}\int_{K_1}|VDM_k(z_{1,1},z_{1,2},...,a_{d,m_{d,k}})|^2d\nu_1(z_{1,1})d\nu_1(z_{1,2}).
\end{equation*}
Continuing in this way, we obtain our result.
\vskip4pt
\noindent {\bf Case III: All coefficients $c_{i,j}$ are real numbers.}
\vskip4pt
This case 
%where all $c_{ij}$ are nonnegative but at least one of the $c_{i,j}$ is irrational 
will follow from the previous case and Proposition \ref{tdlim}. We can assume that $K_1,\ldots,K_d$ are contained in a disk of radius $1/2$ so that all factors 
$|z_{i,l}-z_{j,p}|\leq 1$, $i,j=1,\ldots,d$, $l=1,\ldots,m_{i,k}$, $p=1,\ldots,m_{j,k}$. 
Then for any $\widehat C$ with rational entries 
$\hat c_{i,j}$ as in the proof of Proposition \ref{tdlim},
%$\geq c_{i,j}$, 
we have 
\begin{equation}\label{vdmmon}|VDM_k^{(\widehat C)}({\bf Z_{k}})|\leq |VDM_k^{(C)}({\bf Z_{k}})|,
\end{equation}
for recall that $\hat c_{i,j}\downarrow c_{i,j}$ if $c_{i,j}\geq 0$ and $\hat c_{i,j}\uparrow c_{i,j}$ if $c_{i,j}<0$. Hence
$$\delta^{\widehat C}(K)=\lim_{k\to \infty}\widehat Z_k^{\frac{|r|^2}{|m_k|(|m_k|-1)}}\leq \liminf_{k\to \infty} Z_k^{\frac{|r|^2}{|m_k|(|m_k|-1)}}\leq \limsup_{k\to \infty} Z_k^{\frac{|r|^2}{|m_k|(|m_k|-1)}}\leq \delta^C(K).$$
From Proposition \ref{tdlim} we have 
$$\lim_{\widehat C \to C}\delta^{\widehat C}(K)=\delta^{C}(K)$$
which finishes the proof in the unweighted case.
% if all $c_{ij}$ are nonnegative. 
%\vskip4pt
%\noindent {\bf Case II: $c_{12}<0$}.
%\vskip4pt
%{\red If $c_{12}<0$, we repeat the above steps using the rational Bernstein-Markov property. The $c_{ij}$ integer and rational cases follow directly; the only extra observation necessary to mention is that if each $c_{i,j}$ is a nonnegative integer, then $VDM_k(z_1,...,z_{m_k},w_1,...,w_{n_k})$ is a rational function whose numerator is of degree at most $\max[m_kc_{1,1},n_kc_{2,2}]$ in each variable while the denominator is of degree at most $c_{1,2}\max[m_k,n_k]$ in each variable. Thus, e.g., in the case where all $c_{ij}$ are integers, we replace ``$m_k+n_k$'' in the applications of the rational Bernstein-Markov property by ``$\max[m_k,n_k]$''. We also need to slightly modify the argument in the case where one of the $c_{i,j}$ is irrational: here, since (see the proof of Proposition \ref{tdlim}) we may choose $\widehat C$ with rational entries $\hat c_{i,j}$ approximating $C$ and with $\hat c_{1,2}\leq c_{1,2}$ (but $\hat c_{i,i}\geq c_{i,i}$ for $i=1,2$), for $K_1,K_2$ contained in a disk of radius $1/2$ we still have the inequality (\ref{vdmmon}).
%} 
\end{proof}

Fix a tuple of weights $Q$. Given $\nu$ as in Proposition \ref{weightedtd}, i.e., so that $(K,\nu,Q)$ satisfies a weighted Bernstein-Markov property, and given a sequence $\{m_k\}$ satisfying (\ref{correct2}), define a probability measure $Prob_k$ on $K^{k}$: for a Borel set $A\subset K^{k}$,
\begin{equation}\label{probk}Prob_k(A):=\frac{1}{Z^Q_k}\cdot \int_A  |VDM_k^Q({\bf Z_k})|^2  d\nu({\bf Z_k}).
\end{equation}
Directly from Proposition \ref{weightedtd} and (\ref{probk}) we obtain the following estimate.

\begin{corollary} \label{largedev} Let $(K,\nu,Q)$ satisfy a weighted Bernstein-Markov property. Given $\eta >0$, define
 \begin{equation}\label{aketa}
 A_{k,\eta}:=\{{\bf Z_k}\in K^{k}: |VDM_k^Q({\bf Z_k})|^2 \geq 
 ( \delta_Q(K) -\eta)^{|m_{k}|(|m_{k}|-1)/|r|^{2}}\}.
 \end{equation}
Then there exists $k^*=k^*(\eta)$ such that for all $k>k^*$, 
$$Prob_k(K^{k}\setminus A_{k,\eta})\leq \Big(1-\frac{\eta}{2 \delta_Q(K)}\Big)^{|m_{k}|(|m_{k}|-1)/|r|^{2}}\nu(K^{k}). %\setminus A_{k,\eta}).
$$
\end{corollary}	
	
	We get the induced product probability measure ${\bf P}$ on the space of arrays on $K$, 
	$$\chi:=\{X=\{{\bf Z_{k}}\in K^{k}\}_{k\geq 1}\},$$ 
	namely,
	$$(\chi,{\bf P}):=\prod_{k=1}^{\infty}(K^{k},Prob_k).$$
As an immediate consequence of the Borel-Cantelli lemma, we obtain: 

\begin{corollary}\label{416} Let $(K,\nu,Q)$ satisfy a weighted Bernstein-Markov property. 
For ${\bf P}$-a.e. array $X \in \chi$, 
$$\mu^k=(\mu_1^k,\ldots,\mu_d^k):=\Big(\frac{r_1}{m_{1,k}}\sum_{j=1}^{m_{1,k}} \delta_{z_{1,j}^{(k)}},\ldots,
\frac{r_d}{m_{d,k}}\sum_{j=1}^{m_{d,k}} \delta_{z_{d,j}^{(k)}}\Big)\to \mu^{K,Q} \ \hbox{weak-* as }k\to\infty.$$
\end{corollary}

\begin{proof} From Proposition \ref{awf} it suffices to verify for ${\bf P}$-a.e. array 
$X=\{{\bf Z_{k}}\}_{k} \in \chi$, 
\begin{equation}\label{borcan}\liminf_{k\to \infty} \Big(|VDM_k^Q({\bf Z_{k}})|\Big)
^{2|r|^{2}/|m_{k}|(|m_{k}|-1)}  =  \delta_Q(K).\end{equation}
Given $\eta >0$, the condition that for a given array $X=\{{\bf Z_{k}}\}_{k}$ we have
$$\liminf_{k\to \infty} \Big(|VDM_k^Q({\bf Z_{k}})|\Big)^{2|r|^{2}/|m_{k}|(|m_{k}|-1)}\leq  \delta_Q(K)-\eta$$
means that ${\bf Z_{k}}\in K^{k}\setminus A_{k,\eta}$ for infinitely many $k$. Thus setting 
$$E_k:=\{X \in \chi: {\bf Z_{k}}\in K^{k}\setminus A_{k,\eta}\},$$
we have
$${\bf P}(E_k)\leq Prob_k(K^{k}\setminus A_{k,\eta})\leq (1-\frac{\eta}{2 \delta_Q(K)})
^{|m_{k}|(|m_{k}|-1)/|r|^{2}}\nu(K^{k}),$$
whence $\sum_{k=1}^{\infty} {\bf P}(E_k)<+\infty$. By the Borel-Cantelli lemma, 
$${\bf P}(\limsup_{k\to \infty} E_k)=0,\quad\text{where}\quad\limsup_{k\to \infty}E_{k}
=\cap_{k=1}^{\infty}\cup_{j=k}^{\infty}E_{j}.$$ 
Thus, with probability one, only finitely many $E_k$ occur,  and (\ref{borcan}) follows.
\end{proof}

\section{Approximation of equilibrium problems with non-admissible weights}
	
In Section \ref{functionals}, we will need to consider equilibrium problems with weights that are the negatives of potentials. These weights, if non-continuous, are non-admissible in the sense given in Section \ref{vep}. The aim of this section is to show that one can approach such equilibrium problems by a sequence of equilibrium problems with continuous weights, see 
Lemma \ref{lemma-scal} for the scalar case and Lemma \ref{keylem} for the vector case. 
In this section, $K$ (or its component sets in the vector setting) will always be nonpolar.
\begin{lemma}\label{lem-non-adm}
Let $\mu\in\MM_r(K)$, $K\subset \C$ compact, $I(\mu)<\infty$. Consider the possibly non-admissible weight $u:=-U^{\mu}$ on $K$. The weighted minimal energy on $K$ is obtained with the measure $\mu$, that is
$$\forall\nu\in\MM_r(K),\quad I(\mu)+2\int ud\mu\leq I(\nu)+2\int ud\nu,$$
with equality if and only if $\nu=\mu$. 
\end{lemma}
\begin{proof}
We may assume that $I(\nu)<\infty$. The inequality may be rewritten as
$$0\leq I(\nu)-2I(\mu,\nu)+I(\mu)=I(\nu-\mu),$$
which is true. Moreover, the energy $I(\nu-\mu)$ can vanish only when $\nu=\mu$ (cf., Lemma I.1.8 in \cite{ST}).
\end{proof}

\begin{lemma}\label{lemma-scal}
Let $K\subset \C$ be compact and nonpolar and let $\mu\in\MM_r(K)$ with $I(\mu)<\infty$. There exist a sequence $\{K_n\}$ of compact subsets of $K$, a sequence of continuous functions $Q_{n}$ on $K$, and a sequence $\{\mu_n\}\subset \MM_r(K)$ such that
\begin{enumerate}
\item each $K_n$ is regular; $K_n\subset K_{n+1}$; and $\cup_n K_n =K\setminus P$ where $P$ is polar;
\item $Q_{n}(z)\downarrow u(z):=-U^{\mu}(z),\quad z\in K$;
\item $\mu_n$ is the weighted energy minimizing measure over $\mathcal M_r(K_n)$ of $K_n,Q_n|_{K_n}$ and $$\tilde V_{Q_n}(z):=-U^{\mu_{n}}(z)+F_{n}\downarrow u(z):=-U^{\mu}(z),\quad z\in\C$$
(defining the notation $\tilde V_{Q_n}$ and the constant $F_n$) where $\tilde V_{Q_n}$ and hence $U^{\mu_{n}}$ are continuous. \end{enumerate}
We have the following properties:\\
(i) The Robin constants $F_{n}$ tend to 0 as $n\to\infty$.\\
(ii) The measures $\mu_{n}$ tend weak-* to $\mu$, as $n\to\infty$.\\
(iii) The energies $I(\mu_{n})$ tend to $I(\mu)$ as $n\to\infty$.\\
\end{lemma}
\begin{proof}
Item 1. follows from Ancona's theorem \cite{ancona}. Precisely, for each $n$ we can find $\tilde K_n\subset K$ regular with $Cap(K\setminus \tilde K_n)<1/n$; then $K_n:=\cup_1^n \tilde K_j$ work. For 2., 
the function $u$ is usc whence the existence of a monotone sequence of continuous functions $Q_{n}$ decreasing to $u$ on $K$.

To prove 3., note first that $Q_{n}|_{K_n}$ are continuous and $K_n$ are regular so $\tilde V_{Q_{n}}$ are continuous on $\C$ (cf., Theorem I.5.1 in \cite{ST}). Since $Q_{n}$ is decreasing on $K$, and $K_n \subset K_{n+1}$, $\tilde V_{Q_{n}}$ is decreasing on $\C$.
%, and
%$$V_{Q_{n}}^{*}=Q_{n}\downarrow u\quad\text{q.e. on}\supp\mu_{n}.$$ 
We have, since $\tilde V_{Q_{n}}=Q_n$ q.e. on supp$\mu_n$,
$$F_{n}=U^{\mu_{n}}(z)+Q_{n}(z)\geq U^{\mu_{n}}(z)-U^{\mu}(z),\quad\text{q.e. }z\in\supp\mu_{n}$$
and $I(\mu_{n})$ is finite, hence by the principle of domination (cf., p. 43 of \cite{ST}),
$$U^{\mu_{n}}(z)\leq U^{\mu}(z)+F_{n},\quad z\in\C.$$
Consequently, $\tilde V_{Q_{n}}$ converges in $\C$ to some subharmonic function
$f\geq -U^{\mu}=u$ on $\C$. Since
$$\forall n\geq 0,\quad u(z)\leq f(z)\leq \tilde V_{Q_{n}}(z)\leq Q_{n}(z),\quad z\in K_n,$$
$Q_{n}$ decreases to $u$ on $K$, and $\cup_n K_n =K\setminus P$ where $P$ is polar, we have that $f=u$ q.e. on $K$. In particular 
$$f(z)\leq u(z),\quad \text{q.e. }z\in K.$$
Again, by the principle of domination (for subharmonic functions), 
$$f(z)\leq u(z),\quad z\in\C.$$
Hence $f=u$ on $\C$, which proves 3.

Since $U^{\mu_{n}}-F_{n}$ tends to $U^{\mu}$ pointwise in $\C$, the fact (i) that $F_{n}$ tends to 0 simply follows from the behavior of potentials of compactly supported positive measures of total mass $r$ at infinity: each such function decays like $-r\log |z|+0(1/|z|)$. Then fact (ii) that $\mu_{n}\to\mu$ weak-* is a consequence of the monotone convergence $U^{\mu_{n}}-F_{n}\uparrow U^{\mu}$ in $\C$ (this would also follow from the stronger convergence in energy (property (iii))). 

For the convergence of energies, we observe that
$$I(\mu_{n})-rF_{n}=\int (U^{\mu_{n}}-F_{n})d\mu_{n}\leq\int U^{\mu}d\mu_{n}
=\int (U^{\mu_{n}}-F_{n})d\mu+rF_{n}\leq I(\mu)+rF_{n}.$$
Hence,
$$\limsup_{n\to\infty} I(\mu_{n})\leq I(\mu).$$
Since we also have that $I(\mu)\leq\liminf_{n\to\infty}I(\mu_{n})$ by the weak-{*} convergence of $\mu_{n}$ to $\mu$, we obtain that $I(\mu_{n})$ tends to $I(\mu)$.
\end{proof}

Next, we give analogs of Lemmas \ref{lem-non-adm} and \ref{lemma-scal} for the vector problem with interaction matrix $C$. 
\begin{lemma} \label{5.3}
Let $\mu=(\mu_{1},\ldots,\mu_{d})\in\MM_r(K)$, $K=(K_{1},\ldots,K_{d})$ a tuple of compact sets, $I(\mu_{i})<\infty$, $i=1,\ldots,d$. Consider the non-admissible weight $u:=(-U^{\mu_{1}},\ldots,-U^{\mu_{d}})$ on $K$. The weighted minimal energy on $K$ is obtained with the measure $\mu$, that is
$$\forall\nu=(\nu_{1},\ldots,\nu_{d})\in\MM_r(K),\quad E_{u}(\mu)\leq E_{u}(\nu),$$
with equality if and only if $\nu=\mu$. 
\end{lemma}
\begin{proof}
For a tuple of weights $Q$, we have that
$$E_{Q}(\nu)-E_{Q}(\mu)=2\sum_{i=1}^{d}\int(U^{\mu}_{i}+Q_{i})d(\nu_{i}-\mu_{i})+E(\nu-\mu).$$
Here, with $Q=u$, we simply get
$$E_{u}(\nu)-E_{u}(\mu)=E(\nu-\mu),$$
and from \cite[Proposition 2.9]{BKMW} (or \cite[Chapter 5]{NS} if $C$ is positive definite) we know that $E(\nu-\mu)$ is nonnegative and can vanish only when $\nu=\mu$.
\end{proof}
Given $\mu=(\mu_{1},\ldots,\mu_{d})\in\MM_r(K)$ with $I(\mu_{i})<\infty$, $i=1,\ldots,d$, we write $u_i:=-U^{\mu_{i}}$ and $u:=(-U^{\mu_{1}},\ldots,-U^{\mu_{d}})$ as above. By Lemma \ref{lemma-scal}, we know that, for each $i$, there exists a sequence of continuous functions $Q_{n,i}$ defined on $K_{i}$ and measures $\mu_{n,i}$ on $K_i$ with $\mu_{n,i}\to \mu_i$ weak-* such that, as $n\to\infty$,
\begin{equation*}%\label{conv-V}
Q_{n,i}(z)\downarrow u_{i}(z):=-U^{\mu_{i}}(z),\quad z\in K_{i},\quad\text{and }\quad
-U^{\mu_{n,i}}(z)+F_{n,i}\downarrow u_i(z),\quad z\in\C,
\end{equation*}
where $F_{n,i}\to 0$ as $n\to \infty$;  $U^{\mu_{n,i}}$ are continuous; and $I(\mu_{n,i})\to I(\mu_i)$. 
\begin{lemma}\label{keylem}
Given $\mu=(\mu_{1},\ldots,\mu_{d})\in\MM_r(K)$ with $I(\mu_{i})<\infty$, $i=1,\ldots,d$, let $\mu^{(n)}=(\mu_{n,1},\ldots,\mu_{n,d})$ as above. The following holds true:\\
(i) The tuple of measures $\mu^{(n)}$ tends (component-wise) weak-* to the tuple of measures $\mu$.\\
(ii) The unweighted energy $E(\mu^{(n)})$ tends to $E(\mu)$ as $n\to\infty$.\\
(iii) The tuple of weights $\widehat Q_{n}$ such that
$$\widehat Q_{n,i}(z)=-\sum_{j=1}^{d} c_{i,j}U^{\mu_{n,j}}(z),\quad z\in K_{i},\quad i=1,\ldots,d,$$
are continuous and the tuple of measures $\mu^{(n)}$ is extremal for the vector problem with interaction matrix $C$ and weight $\widehat Q_{n}$.
Moreover, $E_{\widehat Q_{n}}(\mu^{(n)})$ tends to $E_{u}(\mu)$ as $n\to\infty$. \end{lemma}
\begin{proof}
By Lemma \ref{lemma-scal}, each component of $\mu^{(n)}$ tends to the corresponding component of $\mu$. We also know that $I(\mu_{n.i})$ tends to $I(\mu_{i})$ for $i=1,\ldots,d$. For the mutual energies, we have
\begin{align*}
I(\mu_{n,i},\mu_{n,j})-r_{j}F_{n,i} & =\int (U^{\mu_{n,i}}-F_{n,i})d\mu_{n,j}  \leq\int U^{\mu_{i}}d\mu_{n,j}\\
& =\int (U^{\mu_{n,j}}-F_{n,j})d\mu_{i}+r_{i}F_{n,j}  \leq I(\mu_{j},\mu_{i})+r_{i}F_{n,j}.
\end{align*}
Using the fact that $F_{n,i}\to 0$ as $n\to \infty$, 
$$\limsup_{n\to\infty} I(\mu_{n,i},\mu_{n,j})\leq I(\mu_{i},\mu_{j}).$$
Since the mutual energies are also lowersemicontinuous,
%we also have that $I(\mu)\leq\liminf_{n\to\infty}I(\mu_{n})$ by the weak-{*} convergence of $\mu_{n}$ to $\mu$, 
we obtain that 
\begin{equation}\label{lim-I}
I(\mu_{n,i},\mu_{n,j})\to I(\mu_{i},\mu_{j}),\quad\text{ as }n\to\infty,
\end{equation}
which shows assertion (ii). For assertion (iii), $\widehat Q_{n}$ is continuous because, for each $i$,  the potential $U^{\mu_{n,i}}$ is continuous. The tuple $\mu^{(n)}$ is extremal for $\widehat Q_{n}$ because the variational inequalities characterizing the solution of the equilibrium problem (see (\ref{eq-var1})-(\ref{eq-var2})) are trivially satisfed. For the convergence of the energies $E_{\widehat Q_{n}}(\mu^{(n)})$ to $E_{u}(\mu)$, it remains to check that, for all $i$,
$$\int\widehat Q_{n,i}d\mu_{n,i}\to-\int U^{\mu}_{i}d\mu_{i},\quad\text{ as }n\to\infty,$$
which is (\ref{lim-I}) with $i=j$.
%true since we already proved that, for all $i,j$,
%$$\int U^{\mu_{n,j}}d\mu_{n,i}\to\int U^{\mu_{j}}d\mu_{i}.$$
\end{proof}

\section{The vector energy functionals}\label{functionals}
In this section we define $L^{\infty}$ vector energy functionals $\overline W,\underline W$ and weighted versions $\overline W_Q,\underline W_Q$, as well as $L^2$ vector energy functionals $\overline J,\underline J$ and weighted versions $\overline J_Q,\underline J_Q$ using (weighted) Bernstein-Markov measures. 

We proceed with the definitions. Fix $K=(K_1,\ldots,K_d)$, $r_1,\ldots, r_d >0$, an interaction matrix $C\geq 0$, and a strong Bernstein-Markov measure $\nu=(\nu_1,\ldots,\nu_d)$; again, as in Remark \ref{caution} this Bernstein-Markov property is taken to be with respect to polynomials if all $c_{i,j}\geq 0$ and with respect to rational functions otherwise. Fix a sequence $\{m_k\}$ satisfying (\ref{correct2}). Given $G\subset {\mathcal M}_r(K)$, for each $k=1,2,...$ we set
\begin{equation}\label{nbhddef}
\tilde G_k:= \Big\{{\bf a} =(a_{1,1},...,a_{1,m_{1,k}}, a_{2,1},...,a_{d,m_{d,k}})\in K^{k}: 
\Big(\frac{r_1}{m_{1,k}}\sum_{j=1}^{m_{1,k}} \delta_{a_{1,j}},\ldots,\frac{r_d}{m_{d,k}}\sum_{j=1}^{m_{d,k}} \delta_{a_{d,j}}\Big)\in G\Big\}
\end{equation}
and define
$$W_k(G):=\sup \{ |VDM_k({\bf a})|^{2|r|^2/|m_k|(|m_k|-1)}: {\bf a} \in \tilde G_{k}\}$$
and
$$J_k(G):=\Big[\int_{\tilde G_{k}}|VDM_k({\bf a})|^{2}d\nu ({\bf a})\Big]
^{|r|^2/|m_k|(|m_k|-1)}.$$
\begin{definition} \label{jwmu} For $\mu \in {\mathcal M}_r(K)$ we define
$$\overline J(\mu):=\inf_{G \ni \mu} \overline J(G) \ \hbox{where} \ \overline J(G):=\limsup_{k\to \infty} J_k(G);$$
$$\underline J(\mu):=\inf_{G \ni \mu} \underline J(G) \ \hbox{where} \ \underline J(G):=\liminf_{k\to \infty} J_k(G);$$
and
$$\overline W(\mu):=\inf_{G \ni \mu} \overline W(G) \ \hbox{where} \ \overline W(G):=\limsup_{k\to \infty} W_k(G);$$
$$\underline W(\mu):=\inf_{G \ni \mu} \underline W(G) \ \hbox{where} \ \underline W(G):=\liminf_{k\to \infty} W_k(G).$$
\end{definition}

\noindent Here the infima are taken over all neighborhoods $G$ of the measure $\mu$ in ${\mathcal M}_r(K)$ with the weak-* topology.

%{\blue It {\sl should} follow from the definition of $W$ in the vector case that if either $I(\mu_1)=\infty$ or $I(\mu_2)=\infty$, we have $W(\mu)=0$. Indeed, from (\ref{useforw}), we have $$|VDM_k(z_0,...,z_k,w_0,...,w_k)|$$$$\leq |VDM_k^S(z_0,...,z_k)|^{c_{1,1}}\cdot |VDM_k^S(w_0,...,w_k)|^{c_{2,2}}\cdot D(K_1,K_2)^{c_{1,2}(k+1)^2};$$from the definition of $W_k(\mu,G)$, taking $G=G_1\cup G_2$ where $G_i$ is a neighborhood of $\mu_i$, we should have$$W_k(\mu,G)\leq W_k(\mu_1,G_1)^{c_{1,1}/4}\cdot W_k(\mu_2,G_2)^{c_{2,2}/4}\cdot D(K_1,K_2)^{c_{1,2}/2}$$(or something close to this) which should lead to the following: {\it if $W(\mu_1)=0$ or $W(\mu_2)=0$, then $W(\mu)=0$; i.e., if $I(\mu_1)=\infty$ or $I(\mu_2)=\infty$, we have $W(\mu)=0$.}}

 Note that $\overline W,\underline W$ are independent of $\nu$ but, a priori, $\overline J,\underline J$ depend on $\nu$. The weighted versions of these functionals are defined for admissible $Q$ starting with
$$W^Q_k(G):=\sup \{ |VDM^Q_k({\bf a})|^{2|r|^2/|m_k|(|m_k|-1)}: {\bf a} \in \tilde G_{k}\} \ \hbox{and}$$
\begin{equation}\label{jkqmu}
J^Q_k(G):=\Big[\int_{\tilde G_{k}}|VDM^Q_k({\bf a})|^{2}d\nu ({\bf a})\Big]
^{|r|^2/|m_k|(|m_k|-1)}.
\end{equation}
\begin{definition} \label{jwmuq} For $\mu \in \mathcal M_r(K)$ we define
$$\overline J^Q(\mu):=\inf_{G \ni \mu} \overline J^Q(G) \ \hbox{where} \ \overline J^Q(G):=\limsup_{k\to \infty} J^Q_k(G);$$
$$\underline J^Q(\mu):=\inf_{G \ni \mu} \underline J^Q(G) \ \hbox{where} \ \underline J^Q(G):=\liminf_{k\to \infty} J^Q_k(G);$$
and
$$\overline W^Q(\mu):=\inf_{G \ni \mu} \overline W^Q(G) \ \hbox{where} \ \overline W^Q(G):=\limsup_{k\to \infty} W^Q_k(G);$$
$$\underline W^Q(\mu):=\inf_{G \ni \mu} \underline W^Q(G) \ \hbox{where} \ \underline W^Q(G):=\liminf_{k\to \infty} W^Q_k(G).$$
\end{definition}

\noindent Again the infima are taken over all neighborhoods $G$ of the measure $\mu$ in ${\mathcal M}_r(K)$.

The idea behind the $\overline W,\underline W$ (or $\overline W^Q,\underline W^Q$) functionals comes from the definition of the (weighted) transfinite diameter in Proposition \ref{eetd}. Given $\mu$, we consider all sequences of discrete measures associated to ${\bf a}={\bf a}^{k}\in K^k$ of the form 
$$\mu^k:=(\frac{r_1}{m_{1,k}}\sum_{j=1}^{m_{1,k}} \delta_{a_{1,j}},\ldots,\frac{r_d}{m_{d,k}}\sum_{j=1}^{m_{d,k}} \delta_{a_{d,j}})$$
with $\mu^k \to \mu$ weak-* and we maximize the asymptotic behavior of the corresponding sequence of numbers $\{|VDM_k({\bf a})|^{2|r|^2/|m_k|(|m_k|-1)}\}$ (or $\{|VDM^Q_k({\bf a})|^{2|r|^2/|m_k|(|m_k|-1)}\}$) over all such $\{\mu^k\}$. The $\overline J,\underline J$ (or $\overline J^Q,\underline J^Q$) functionals utilize $L^2(\nu)-$averages instead. Note that if $\mu^k \to \mu$ weak-*, then given any neighborhood $G\subset {\mathcal M}_r(K)$ of $\mu$, the tuple of points ${\bf a}={\bf a}^{k}$ belongs to $\tilde G_k$ for all $k$ sufficiently large.
 
All the functionals are uppersemicontinuous  on ${\mathcal M}_r(K)$ in the weak-* topology. We write
$$\int_K Qd\mu:= \sum_{i=1}^{d}\int_{K_i} Q_id\mu_i.$$ 
Then the following properties hold (and with the $\underline J,\underline J^Q, \underline W, \underline W^Q$ functionals as well): 
\begin{enumerate}
\item $\overline J^Q(\mu)\leq \overline W^Q(\mu)\leq \delta_Q (K)$ for admissible $Q$;
\item $ \overline W(\mu)=\overline W^Q(\mu)\cdot e^{2\int_K Qd\mu}
\ \hbox{and} \ 
\overline J(\mu)=\overline J^Q(\mu)\cdot e^{2\int_K Qd\mu}$  for $Q$ {\it continuous}. 
\end{enumerate}
\begin{proof}[Proof of 2.]
First we observe that if $\mu \in {\mathcal M}_r(K)$ and $Q$ is continuous on $K$,  given $\epsilon >0$, there exists a neighborhood $G \subset  {\mathcal M}_r(K)$ of $\mu$ with
$$\big|\sum_{i=1}^{d}\int_{K_i} Q_i\Big(d\mu_i -\frac{r_i}{m_{i,k}}\sum_{j=1}^{m_{i,k}} 
\delta_{a_{i,j}}\Big)\big|\leq\epsilon \quad \hbox{for} \ {\bf a} \in \tilde G_k$$
for $k$ sufficiently large. Thus we have
$$-\epsilon -\int_{K}Qd\mu\leq -\sum_{i=1}^{d}\frac{r_i}{m_{i,k}}\sum_{j=1}^{m_{i,k}}
 Q_{i}(a_{i,j})
 \leq \epsilon-\int_{K}Qd\mu.$$
%Thus {$$
%e^{-(m_k/r_1)^2\epsilon}\leq \prod_{j=1}^{m_k} e^{-(m_k/r_1)Q_1(a_j)} \Big(e^{\int_{K_1} Q_1d\mu_1}\Big)^{(m_k/r_1)^2}\cdot \prod_{j=1}^{n_k} e^{-(m_k/r_1)Q_2(\alpha_j)}\Big(e^{\int_{K_2} Q_2d\mu_2}\Big)^{(m_k/r_1)^2} $$
%$$\leq  e^{(m_k/r_1)^2\epsilon}.$$
Recalling (\ref{correct3}) we get that
\begin{equation}\label{ineg}
-\alpha_{k}(\epsilon +\int_{K}Qd\mu)\leq -\frac{|r|^{2}}{|m_{k}|(|m_{k}|-1)}\sum_{i=1}^{d}
\frac{m_{i,k}}{r_{i}}\sum_{j=1}^{m_{i,k}}
 Q_{i}(a_{i,j})
 \leq \beta_{k}(\epsilon-\int_{K}Qd\mu),
 \end{equation}
 where $\alpha_{k}$ and $\beta_{k}$ tend to 1 as $k$ tends to infinity.
%$$\frac{r_1^2}{m_k^2}\asymp \frac{r_2^2}{n_k^2} \asymp \frac{2r_1r_2}{m_kn_k},$$
%we have
%$$e^{-(m_k/r_1)^2\epsilon}\lessapprox \prod_{j=1}^{m_k} e^{-(m_k/r_1)Q_1(a_j)} \prod_{j=1}^{n_k} e^{-(n_k/r_2)Q_2(\alpha_j)}\cdot e^{\frac{2m_k^2n_k^2}{n_k^2r_1^2+m_k^2r_2^2}[\int_{K_1} Q_1d\mu_1+\int_{K_2} Q_2d\mu_2]}$$
%$$\lessapprox  e^{(m_k/r_1)^2\epsilon}.$$
Since
$$|VDM_k^Q({\bf a})|:=|VDM_k({\bf a})|\cdot \prod_{i=1}^{d}\prod_{j=1}^{m_{i,k}} 
e^{-\frac{m_{i,k}}{r_i}Q_i(a_{i,j})},$$ 
%writing 
%$$\int_K Qd\mu:= \int_{K_1} Q_1d\mu_1+\int_{K_2} Q_2d\mu_2,$$
%and observing that
%$$(m_k/r_1)^2\approx \frac{2m_k^2n_k^2}{n_k^2r_1^2+m_k^2r_2^2},$$
we deduce from (\ref{ineg}) that
$$|VDM_k({\bf a})|e^{-\alpha_{k}\frac{|m_{k}|(|m_{k}|-1)}{|r|^{2}}
(\epsilon+\int_K Qd\mu)} 
\leq |VDM^Q_k({\bf a})| \leq 
|VDM_k({\bf a})|e^{\beta_{k}\frac{|m_{k}|(|m_{k}|-1)}{|r|^{2}}(\epsilon-\int_K Qd\mu)}.
$$
Now we take the supremum over ${\bf a}\in \tilde G_k$ and take a 
$|m_{k}|(|m_{k}|-1)/2|r|^{2}$-th root of each side to get
$$W_k(G)e^{-2\alpha_{k}(\epsilon+\int_K Q d\mu)} \leq W^Q_k(G) \leq W_k(G)
e^{2\beta_{k}(\epsilon-\int_K Q d\mu)}.
$$
Precisely, given $\epsilon >0$, these inequalities are valid for $G$ a sufficiently small neighborhood of $\mu$. Hence we get, upon taking $\limsup_{k\to \infty}$, the infimum over $G\ni \mu$, and noting that $\epsilon >0$ is arbitrary,
$$\overline W(\mu)=\overline W^Q(\mu)\cdot e^{2\int_K Qd\mu}$$
as desired. The proof that $\overline J(\mu)=\overline J^Q(\mu)\cdot e^{2\int_K Qd\mu}$ is similar.
\end{proof}

Note from the definition of $E$ and $E_Q$ we have a similar (obvious) relation
\begin{equation}\label{eprop}E_Q(\mu)=E(\mu)+2\int_KQd\mu.\end{equation}
In particular, $E_Q(\mu^{K,Q})=E(\mu^{K,Q})+2\int_KQd\mu^{K,Q}$ so that, using Proposition  \ref{eetd},
\begin{equation}\label{eprop2}-E(\mu^{K,Q})=\log \delta_Q(K)+2\int_KQd\mu^{K,Q}.\end{equation}
Also, from Proposition \ref{wup},
\begin{equation}\label{wupbound}
\log \underline W(\mu)\leq \log \overline W(\mu)\leq -E(\mu).
\end{equation}
We show equality holds in this last relation. 

 \begin{theorem} \label{wobsolete} Let $K=(K_1,\ldots,K_d)$ be nonpolar and $Q=(Q_1,\ldots,Q_d)$ continuous. Then for any $\mu\in \mathcal M_r(K)$, 
 \begin{equation}\label{minwunwtd}\log \overline W(\mu)=\log \underline W(\mu)=-E(\mu) \ \hbox{and} \end{equation}
 \begin{equation}\label{minwwtd}\log \overline W^Q(\mu)=\log \underline W^Q(\mu)=-E_Q(\mu).\end{equation}
   \end{theorem}

\begin{proof} It suffices to prove (\ref{minwunwtd}) as then (\ref{minwwtd}) follows from property 2 and (\ref{eprop}). We have from (\ref{wupbound}), (\ref{eprop}), property 2 and Proposition  \ref{eetd}, for any $\mu$ and any $Q$, the upper bound (inequality) in (\ref{minwwtd}) and hence in (\ref{minwunwtd}):
\begin{equation}\label{quppb}
 \log \overline W^Q(\mu)\leq -E_Q(\mu)\leq -E_Q(\mu^{K,Q})=\log \delta_Q(K).
 \end{equation}
 In particular, from 2., for any $\mu$ we have
$$
 \log \overline W(\mu)\leq \inf_{Q}\big[\log \delta_Q(K)+2\int_KQd\mu\big].
$$
It turns out that equality holds in this last relation (although we will not need/use this).

To get a lower bound on $\log \overline W(\mu)$, we begin with the case where $\mu=\mu^{K,v}$ for some $v\in  C(K)$. Using Proposition \ref{awf}, if we consider arrays of points $\{{\bf Z}_k\}\subset K$ as in (\ref{Zk}) for which
$$\lim_{k\to \infty} |VDM_k^{v}({\bf Z}_k)|^{2|r|^2/|m_k|(|m_k|-1)} = \delta_v(K),$$
we have 
$$\Big(\frac{r_1}{m_{1,k}}\sum_{j=1}^{m_{1,k}} \delta_{z_{1,j}^{(k)}},\ldots,
\frac{r_d}{m_{d,k}}\sum_{j=1}^{m_{d,k}} \delta_{z_{d,j}^{(k)}}\Big)\to \mu^{K,v}$$ 
weak-*. Thus for any neighborhood $G$ of $\mu^{K,v}$ we have $ \delta_v(K)\leq  \underline W^{v}(G)$; hence 
\begin{equation}\label{newstep}  \underline W^{v}(\mu^{K,v})=\overline W^{v}(\mu^{K,v}) =  \delta_v(K). \end{equation}
Applying 2, (\ref{newstep}) and (\ref{eprop2}) we obtain (\ref{minwunwtd}) for $\mu=\mu^{K,v}$:
\begin{align}\label{wnewstep}\log \underline W(\mu^{K,v})&=\log \overline W(\mu^{K,v}) = \log \overline W^v(\mu^{K,v})+2\int_Kvd\mu^{K,v} \nonumber \\
&=\log \delta_v(K) +2\int_K v d\mu^{K,v}=-E(\mu^{K,v}).\end{align}

Next we take a tuple of measures $\mu\in \mathcal M_r(K)$ with $I(\mu_{i})<\infty$, $i=1,\ldots,d$.  Using Lemma \ref{keylem}, for each $i=1,\ldots,d$, there exists a sequence of continuous functions $Q_{n,i}$ defined on $K_{i}$ and measures $\mu_{n,i}$ such that, as $n\to\infty$,
\begin{equation*}%\label{conv-V}
Q_{n,i}(z)\downarrow u_{i}(z):=-U^{\mu_{i}}(z),\quad z\in K_{i},\quad\text{and }\quad
V^{*}_{Q_{n,i}}(z):=-U^{\mu_{n,i}}(z)+F_{n,i}\downarrow u_i(z),\quad z\in\C.
\end{equation*} 
Here the functions $U^{\mu_{n,i}}$ are continuous; $F_{n,i}\to 0$; and $I(\mu_{n,i})\to I(\mu_i)$. Moreover, writing $\mu^{(n)}=(\mu_{n,1},\ldots,\mu_{n,d})$, from (ii) of the lemma, 
\begin{equation}\label{from2}\lim_{n\to \infty}E(\mu^{(n)}) =E(\mu), 
\end{equation}
and from (iii) of the lemma, for the sequence of continuous functions $\widehat Q_{n}=(\widehat Q_{n,1},\ldots,\widehat Q_{n,d})$ where $\widehat Q_{n,i}(z)=-\sum_{j=1}^{d} c_{i,j}U^{\mu_{n,j}}(z)$ we have $\mu^{(n)}=\mu^{K,\widehat Q_{n}}$. Thus we can apply the previous case   to conclude
$$\log \overline W(\mu^{(n)})= \log  \underline W(\mu^{(n)})=-E(\mu^{(n)}).$$
From uppersemicontinuity of the functional $\mu \to \underline W(\mu)$,
$$\limsup_{n\to \infty} \log \underline W(\mu^{(n)})=\limsup_{n\to \infty} \log \overline W(\mu^{(n)})=\limsup_{n\to \infty} [-E(\mu^{(n)})]\leq \log\underline W(\mu).$$
But from (\ref{from2}) we see that the limit exists and 
$$\lim_{n\to \infty} \log \underline W(\mu^{(n)})=\lim_{n\to \infty}[-E(\mu^{(n)})]=-E(\mu)\leq \log \underline W(\mu).$$
Together with (\ref{wupbound}) we have
$$\log \underline W(\mu)=\log  \overline W(\mu)=-E(\mu).$$

To finish the proof, we must show that if $\mu\in \mathcal M_r(K)$ satisfies $I(\mu_{i})=\infty$ for some $i=1,\ldots,d$, then $E(\mu)=\infty$ and $\overline W(\mu)=0$. The fact that $E(\mu)=\infty$ is clear; then the upper bound in (\ref{wupbound}) shows that $\overline W(\mu)=0$. 
 \end{proof}

We now consider the $\overline J,\underline J, J^Q$ and $\underline J^Q$ functionals. 

 \begin{theorem} \label{obsolete} Let $K=(K_1,\ldots,K_d)$ be nonpolar and $Q=(Q_1,\ldots,Q_d)$ continuous and let $\nu\in {\mathcal M}_1(K)$ satisfy a strong Bernstein-Markov property. Then for any $\mu\in \mathcal M_r(K)$, 
 \begin{equation}\label{minunwtd}\log \overline J(\mu)=\log \overline W(\mu)= \log\underline  J(\mu)=\log \underline W(\mu)=-E(\mu)\end{equation}
and
 \begin{equation}\label{minwtd}\log \overline J^Q(\mu)=\log \overline W^Q(\mu)= \log \underline J^Q(\mu)=\log \underline W^Q(\mu)=-E_Q(\mu).\end{equation}
   \end{theorem}
 \begin{proof} As in the previous proof, it suffices to show (\ref{minunwtd}) since (\ref{minwtd}) follows from property 2. We have the upper bound as before; for the lower bound, we consider the case where $\mu=\mu^{K,v}$ for $v\in  C(K)$. We show the analogue of (\ref{wnewstep}) for $\overline J,\underline J$: 
\begin{equation}\label{jversion}\log \overline J(\mu^{K,v})=\log \underline J(\mu^{K,v})=\log \delta_v (K) +2\int_K vd\mu^{K,v}.\end{equation}
Then (\ref{jversion}) will imply that 
$$\log \overline J(\mu^{K,v})=\log \overline W(\mu^{K,v})= \log\underline  J(\mu^{K,v})=\log \underline W(\mu^{K,v})=-E(\mu^{K,v})$$
and hence 
$$\log \overline J(\mu)=\log \overline W(\mu)= \log\underline  J(\mu)=\log \underline W(\mu)=-E(\mu)$$
for arbitrary $\mu \in  \mathcal M _r(K)$ following the proof of Theorem \ref{wobsolete}. This proves (\ref{minunwtd}). To prove (\ref{jversion}), we first verify the following.
 \medskip
 
\noindent {\sl Claim: Fix a neighborhood $G$ of $\mu^{K,v}$. For $\eta >0$, define $A_{k,\eta}$ as in (\ref{aketa}) with $Q=v$. Given a sequence $\{\eta_j\}$ with $\eta_j\downarrow 0$, there exists a $j_{0}$ and a $k_{0}$ such that
%for any subsequence of positive integers $\{k_j\}$ increasing sufficiently rapidly, 
%we have
%\begin{equation}\label{seq}(1-\frac{\eta_j}{2\delta_v(K)})^{|m_{k_j}|(|m_{k_j}|-1)/|r|^{2}}\to 0 \ \hbox{as} \ j\to \infty;\end{equation}
%\begin{equation}\label{probk2}Prob_{k_j}(K^{k_j}\setminus A_{k_j,\eta_j})\leq (1-\frac{\eta_j}{2\delta_v(K)})^{|m_{k_j}|(|m_{k_j}|-1)/|r|^{2}} \  \hbox{for all} \ j;\end{equation}
%and 
%we have the inclusion 
\begin{equation}\label{setinclu}  \forall j\geq j_{0},\quad\forall k\geq k_{0},\quad A_{k,\eta_j} \subset \tilde G_{{k}}.
%\quad  
%\hbox{for $k$ and $j$ large enough}. 
\end{equation}
}
 \medskip
 \noindent 
 %By Corollary \ref{largedev} we have (\ref{probk2}) for all $k_j \geq k_j^*(\eta_j)$. 
We prove (\ref{setinclu}) by contradiction: if false, there are sequences $\{k_l\}$ and $\{j_{l}\}$ tending to infinity such that for all $l$ sufficiently large we can find a point ${\bf Z_{k_l}}\in 
 A_{k_l,\eta_{j_{l}}} \setminus \tilde G_{k_l}$. But 
 $$\mu^l:=\Big(({r_1}/{m_{1,k_l}})\sum_{i=1}^{m_{1,k_l}}\delta_{z_{1,i}},\ldots,({r_d}/{m_{d,k_l}})\sum_{i=1}^{m_{d,k_l}}\delta_{z_{d,i}}\Big)\not\in G$$ for $l$ sufficiently large contradicts Proposition \ref{awf} since ${\bf Z_{k_l}}\in A_{k_l,\eta_{j_{l}}}$ and $\eta_{j_l}\to 0$ imply $\mu^l\to \mu^{K,v}$ weak-*. This proves the claim.
  %\Big(\frac{n_k^2r_1^2+m_k^2r_2^2}{2m_k^2n_k^2}\Big)^{1/2}$.}
 \medskip

% We observe that since $(m_k,n_k)$ are positive integers with $m_k,n_k\uparrow \infty$
% $$(1-\frac{1}{2l_k\delta_v(K)})^\frac{2m_k^2n_k^2}{n_k^2r_1^2+m_k^2r_2^2}=(1-\frac{1}{2l_k\delta_v(K)})^{l_k^2}\to 0 \ \hbox{as} \ k \to \infty.$$
Fix a neighborhood $G$ of $\mu^{K,v}$ and a sequence $\{\eta_j\}$ with $\eta_j\downarrow 0$. For $j\geq j_{0}$, choose $k=k_j$ large enough so that the inclusion in (\ref{setinclu}) holds true as well as
\begin{equation}\label{probk2}Prob_{k_j}(K^{k_j}\setminus A_{k_j,\eta_j})\leq \Big(1-\frac{\eta_j}{2\delta_v(K)}\Big)^{|m_{k_j}|(|m_{k_j}|-1)/|r|^{2}}\nu(K^{k_{j}}),
\end{equation}
and
\begin{equation}\label{seq}\Big(1-\frac{\eta_j}{2\delta_v(K)}\Big)^{|m_{k_j}|(|m_{k_j}|-1)/|r|^{2}}\nu(K^{k_{j}})\to 0 \quad \hbox{as} \quad j\to \infty,
\end{equation}
which is possible (for (\ref{probk2}) we make use of Corollary \ref{largedev}). In view of (\ref{setinclu}), (\ref{probk}) and (\ref{probk2}), we have
\begin{align}\notag
\frac{1}{Z^v_{k_j}} 
\int_{\tilde G_{k_j}}  |VDM_{k_j}^{v}({\bf Z_{k_j}})|^2  d\nu({\bf Z_{k_j}})
& \geq \frac{1}{Z^v_{k_j}} 
\int_{A_{k_j,\eta_j} }  |VDM_{k_j}^{v}({\bf Z_{k_j}})|^2  d\nu({\bf Z_{k_j}}) \\
\label{ineg-for-J}
& \geq 1- \Big(1-\frac{\eta_j}{2\delta_v(K)}\Big)^{|m_{k_j}|(|m_{k_j}|-1)/|r|^{2}}\nu(K^{k_{j}}).%\to 1 \ \hbox{as} \ k\to \infty
\end{align}
Note that, because of (\ref{seq}), the lower bound in (\ref{ineg-for-J}) tends to 1 as $j\to\infty$. Then,
since $\nu$ satisfies a strong Bernstein-Markov property, we derive, along with Proposition \ref{weightedtd}, that
$$\liminf_{j\to \infty} \frac{|r|^{2}}{|m_{k_j}|(|m_{k_j}|-1)}\log \int_{\tilde G_{k_j}}  |VDM_{k_j}^{v}({\bf Z_{k_j}})|^2  d\nu({\bf Z_{k_j}}) \geq \log \delta_v(K).$$
Giving any sequence of positive integers $\{k\}$ we can find a subsequence $\{k_j\}$ as above corresponding to some $\eta_j\downarrow 0$; hence
$$\liminf_{k\to \infty} \frac{|r|^{2}}{|m_{k}|(|m_{k}|-1)}\log \int_{\tilde G_{k}}  |VDM_{k}^{v}({\bf Z}_k)|^2 d\nu({\bf Z}_{k})\geq \log \delta_v(K).$$
It follows that 
$$
\log \underline J^{v}(G)\geq \log \delta_v(K).$$
Taking the infimum over all neighborhoods $G$ of $\mu^{K,v}$ we obtain
$$\log \underline J^{v}(\mu^{K,v})\geq  \log \delta_v(K).$$
Thus we have the version of (\ref{newstep}) with $\overline J^{v}$ and $\underline J^{v}$:
\begin{equation}\label{jeqn}\log \underline J^{v}(\mu^{K,v})=\log \overline J^{v}(\mu^{K,v})= \log \delta_v(K).\end{equation}
Using 2. with $\mu = \mu^{K,v}$, from (\ref{jeqn}) we obtain (\ref{jversion}).
 \end{proof}
 
 \begin{remark} \label{fourfour} The equality of $\overline J^Q$ and $\underline J^Q$ is the basis for the proof of our large deviation principle in the next section. From now on, we simply use the notation $J,J^Q,W,W^Q$ without the overline or underline. Note that, in particular, these functionals are independent of the sequence $\{m_k\}$ satisfying (\ref{correct2}); and $J,J^Q$ are independent of the strong Bernstein-Markov measure $\nu$.  \end{remark}

\section{Large deviation principle}
In this section, $K=(K_1,\ldots,K_d)$ are nonpolar disjoint compact sets in $\C$. We fix an interaction matrix $C\geq 0$, positive numbers $r_1,\ldots,r_d$, as well as a measure $\nu=(\nu_1,\ldots,\nu_d)$ satisfying a strong Bernstein-Markov property and a tuple of continuous weights $(Q_1,\ldots,Q_d)$. Again, this Bernstein-Markov property is taken to be with respect to polynomials if all $c_{i,j}\geq 0$ and with respect to rational functions otherwise. We take a sequence of tuples of positive integers $\{m_k\}$ satisfying (\ref{correct2}). As before, we associate to a set of points 
$${\bf Z_k}:=(z_{1,1},...,z_{1,m_{1,k}}, z_{2,1},...,z_{d,m_{d,k}})\in K_1^{m_{1,k}}\times\ldots\times K_d^{m_{d,k}}=K^{k}$$ 
the measure 
$$\mu^k:=\Big(\frac{r_1}{m_{1,k}}\sum_{j=1}^{m_{1,k}} \delta_{z_{1,j}^{(k)}},\ldots,
\frac{r_d}{m_{d,k}}\sum_{j=1}^{m_{d,k}} \delta_{z_{d,j}^{(k)}}\Big)\in \mathcal M_r(K).$$
Define 
$j_k:  K^{k} \to \mathcal M_r(K)$ via 
$$j_k({\bf Z_{k}})=\mu^k
.$$
From (\ref{probk}), 
$\sigma_k:=(j_k)_*(Prob_k) $ is a probability measure on $\mathcal M_r(K)$. We can be more precise about this definition. For a Borel set $G\subset \mathcal M_r(K)$,
\begin{equation}\label{sigmak}\sigma_k(G)=\frac{1}{Z^Q_k} \int_{\tilde G_{k}} |VDM_k^Q({\bf Z_k})|^2 d\nu({\bf Z_k})
\end{equation}
where $\tilde G_{k}$ is defined in (\ref{nbhddef}).

\begin{theorem} \label{ldp} The sequence $\{\sigma_k=(j_k)_*(Prob_k)\}$ of probability measures on $\mathcal M_r(K)$ satisfies a {\bf large deviation principle} (LDP) with speed 
$|m_{k}|(|m_{k}|-1)/2|r|^{2}$ and good rate function $\mathcal I:=\mathcal I_{K,Q}$ where
\begin{equation}\label{ratefcnlform}\mathcal I(\mu):=\log J^Q(\mu^{K,Q})-\log J^Q(\mu)=\log W^Q(\mu^{K,Q})-\log W^Q(\mu)=E_Q(\mu)-E_Q(\mu^{K,Q}).\end{equation}
 \end{theorem}
 
\begin{remark} \label{72} For basic notions involving LDP, we refer the reader to \cite{DZ}. Note that for {\it each} sequence of tuples of positive integers 
 $\{m_k\}$ satisfying (\ref{correct2}) and {\it each} strong Bernstein-Markov measure $\nu$ we get an LDP where the speed depends on $m_k$ but the rate function is independent of both $m_k$ and $\nu$.\end{remark}
 
The following is a special case of a basic general existence result for a LDP given in Theorem 4.1.11 in \cite{DZ}.

\begin{proposition} \label{dzprop1} Let $\{\sigma_{\epsilon}\}$ be a family of probability measures on $\mathcal M_r(K)$. Let $\mathcal B$ be a base for the topology of $\mathcal M_r(K)$. For $\mu\in \mathcal M_r(K)$ let
$$\mathcal I(\mu):=-\inf_{\{G \in \mathcal B: \mu \in G\}}\Big(\liminf_{\epsilon \to 0} \epsilon \log \sigma_{\epsilon}(G)\Big).$$
Suppose for all $\mu\in \mathcal M_r(K)$,
$$\mathcal I(\mu)=-\inf_{\{G \in \mathcal B: \mu \in G\}}\Big(\limsup_{\epsilon \to 0} \epsilon \log \sigma_{\epsilon}(G)\Big).$$
Then $\{\sigma_{\epsilon}\}$ satisfies a LDP with rate function $\mathcal I(\mu)$ and speed $1/\epsilon$. 
\end{proposition}

\begin{proof} ({\it of Theorem \ref{ldp}}): As a base $\mathcal B$ for the topology of $\mathcal M_r(K)$, we can, e.g., take all open sets. For $\{\sigma_{\epsilon}\}$, we take the sequence of probability measures $\{\sigma_k\}$ on $\mathcal M_r(K)$ and we take $\epsilon  ={2|r|^2/|m_k|(|m_k|-1)}$. For $G\in \mathcal B$
$$\frac{2|r|^2}{|m_k|(|m_k|-1)}\log \sigma_k(G)= \log J_k^Q(G)
-\frac{2|r|^2}{|m_k|(|m_k|-1)}\log Z^Q_k$$
using (\ref{jkqmu}) and (\ref{sigmak}). From Proposition \ref{weightedtd}, and (\ref{jeqn}) with $Q$, 
$$\lim_{k\to \infty} \frac{2|r|^2}{|m_k|(|m_k|-1)}\log Z^Q_k=\log  \delta_Q(K)= \log J^Q(\mu^{K,Q});$$ and by Theorem \ref{obsolete}, if $E(\mu)<\infty$,
$$\inf_{G \ni \mu} \limsup_{k\to \infty} \log J_k^Q(G)=\inf_{G \ni \mu} \liminf_{k\to \infty} \log J_k^Q(G)=\log J^Q(\mu).$$
If, on the other hand, $E(\mu)=\infty$, then $J(\mu)=W(\mu)=0$ and hence for $G\ni \mu$ 
$$\lim_{k\to \infty} \log J_k^Q(G)=-\infty.$$
Thus by Proposition \ref{dzprop1}, $\{\sigma_k\}$ satisfies an LDP with rate function 
$$\mathcal I(\mu):=\log J^Q(\mu^{K,Q})-\log J^Q(\mu)=E_Q(\mu)-E_Q(\mu^{K,Q})$$
and speed $|m_{k}|(|m_{k}|-1)/2|r|^{2}$. This rate function is good since $\mathcal M_r(K)$ is compact.
\end{proof}

\section{Possibly intersecting sets}

Many of the results in the paper remain valid for nonpolar compact sets $K_1,\ldots,K_d$ that are not necessarily disjoint. We make the standing assumption, as in \cite{BKMW}, that
\\[.5\baselineskip]
{\it  (i) There exists a vector $(y_1,...,y_d)$ in the range of $C$ such that if $K_i\cap K_j\not =\emptyset$, then $y_iy_j >0$.}
\\[.5\baselineskip]
{\it (ii) If $\{i_1,...,i_m\}\subset \{1,2,...,d\}$ are indices such that the $m$ columns $\{C_{i_j}\}_{j=1,...,m}$ of $C$ are linearly dependent, then $\capa(\cap_{j=1}^m K_{i_j})=0$}.
\\[.5\baselineskip]
These assumptions are automatically satisfied if $C$ is positive definite. In section 2, the existence and uniqueness of a minimizing $d$-tuple of measures for the energy $E$ over $\mu \in \mathcal M_r(K)$ and in the weighted case for the energy $E_Q$ is covered in Theorem 1.8 of \cite{BKMW}. Indeed, it is proved that utilizing the partial potentials
$$U^{\mu}_{i}=\sum_{j=1}^{d}c_{i,j}U^{\mu_{j}},\qquad i=1,\ldots,d,$$ 
a measure $\mu$ minimizes the weighted energy $E_{Q}$ if and only if there exist constants $F_{1},...,F_d$ such that the variational inequalities
\begin{align*}
U^{\mu}_{i}(z) +Q_i& \geq F_{i},\quad \text{q.e. }z\in K_{i},\quad i=1,\ldots,d,\\
U^{\mu}_{i}(z) +Q_i& \leq F_{i},\quad \mu_{i}\text{-a.e. } z\in K_{i},\quad i=1,\ldots,d
\end{align*}
hold.

We claim that if, in addition, we assume that

\begin{equation}\label{hyp} c_{i,j} \ \hbox{\bf{is nonnegative if}} \ K_i\cap K_j\not =\emptyset \end{equation} 
%either
%\begin{enumerate}
%\item all coefficients $c_{i,j}$ of the interaction matrix $C$ are nonnegative; or
%\item $c_{i,j}=0$ whenever $K_i\cap K_j \not =\emptyset$ (and $i\not = j$),
%\end{enumerate}
then all of the results in sections 3-7 remain true. In particular, the Angelesco ensembles satisfying (\ref{hyp}) are covered in this setting as are the Nikishin ensembles when the sets $K_i$ and $K_{i\pm1}$ are disjoint. We indicate the minor modifications of the proofs/results needed in these sections with the above hypotheses.

The equality (\ref{ineg-Iij}) will now be replaced by an inequality with $\liminf$:
$$I(\mu_i,\mu_j)\leq \liminf_{k\to \infty} I(\mu_i^k,\mu_j^k)=\liminf_{k\to \infty}\frac{r_ir_j}{m_{i,k}m_{j,k}}\sum_{l=1}^{m_{i,k}}\sum_{p=1}^{m_{j,k}}\log \frac {1}{|z_{i,l}^{(k)}-z_{j,p}^{(k)}|}.$$
This leaves the rest of the proof of Proposition \ref{wup}, and the results in section 3,  unchanged. Note that the scaling result, (\ref{scal}), still holds. 

Hypothesis (\ref{hyp}) obviates the need for any modifications of the (vector) Bernstein-Markov properties in section 4. The result from \cite[Proposition 2.9]{BKMW} used in Lemma \ref{5.3}  that $E(\nu-\mu)$ is nonnegative and can vanish only when $\nu=\mu$ remains true; and since the variational inequalities listed above characterizing the solution of the equilibrium problem remain valid, all of the arguments in section 5 are unaltered.

The results in sections 6 and 7 rest solely on the preliminaries in the previous sections; thus, {\it Theorems \ref{wobsolete}, \ref{obsolete}, and the LDP Theorem \ref{ldp}, remain true for nonpolar compact sets $K_1,\ldots,K_d$ that are not necessarily disjoint provided assumptions (i), (ii) and (\ref{hyp}) are satisfied}.

\obeylines
\texttt{T. Bloom, bloom@math.toronto.edu
University of Toronto, Toronto, Ontario M5S 2E4 Canada
\medskip
N. Levenberg, nlevenbe@indiana.edu
Indiana University, Bloomington, IN 47405 USA
\medskip
F. Wielonsky, wielonsky@cmi.univ-mrs.fr
Universit\'e Aix-Marseille, CMI 39 Rue Joliot Curie
F-13453 Marseille Cedex 20, FRANCE }
\end{document}